\documentclass[11pt]{article}
\usepackage{amsmath}
\usepackage{amssymb,amsbsy,amsthm,dsfont}
\usepackage[T1]{fontenc}
\usepackage{graphicx}
\usepackage[dvipsnames]{xcolor}
\usepackage{float}
\usepackage{caption}
\usepackage{subcaption}
\usepackage{enumerate}
\usepackage{cases}
\usepackage[margin=1in]{geometry}

\usepackage{hyperref}
\hypersetup{
	colorlinks,
	citecolor=blue,
	filecolor=red,
	linkcolor=blue,
	urlcolor=blue
}

\allowdisplaybreaks[1]

\DeclareMathOperator{\R}{\mathbb{R}} % Real numbers
\DeclareMathOperator{\N}{\mathbb{N}} % Natural numbers
\newcommand{\eps}{\varepsilon} % Epsilon
\newtheorem{theorem}{Theorem}%[section]
\newtheorem{lemma}[theorem]{Lemma}
\newtheorem{proposition}[theorem]{Proposition}

\newtheorem{corollary}[theorem]{Corollary}

\newtheorem{question}[theorem]{Question}

\newtheorem{remark}[theorem]{Remark}

\begin{document}

\title{Finding cliques and dense subgraphs using edge queries}
\author{
	Endre Cs\'oka\thanks{Alfr\'ed R\'enyi Institute of Mathematics; %Supported by the NRDI grant KKP~138270; 
    \url{csokaendre@gmail.com}.}
    \and
    Andr\'as Pongr\'acz\thanks{Alfr\'ed R\'enyi Institute of Mathematics; %Supported by the NRDI grant KKP~138270;
    \url{andras.pong@gmail.com}.}
}
\date{\today}

\maketitle

%%%%%%%%%%%%%%%%
%%% Abstract %%%
%%%%%%%%%%%%%%%%

\begin{abstract}
We consider the problem of finding a large clique in an Erd\H{o}s--R\'enyi random graph where we are allowed unbounded computational time but can only query a limited number of edges. 
Recall that the largest clique in $G \sim G(n,1/2)$ has size roughly $2\log_{2} n$. 
Let $\alpha_{\star}(\delta,\ell)$ be the supremum over $\alpha$ such that there exists an algorithm that makes $n^{\delta}$ queries in total to the adjacency matrix of $G$, in a constant $\ell$ number of rounds, 
and outputs a clique of size $\alpha \log_{2} n$ with high probability. 
We give improved upper bounds on $\alpha_{\star}(\delta,\ell)$ for every $\delta \in [1,2)$ and $\ell \geq 3$. 
We also study analogous questions for finding subgraphs with density at least $\eta$ for a given $\eta$, and prove corresponding impossibility results. 
%The upper bounds for the Maximum Clique Query Problem and the Maximum Dense Subgraph Query Problem are both strictly monotone increasing in $\delta \in [1,2)$ and attain the maximum size of a clique or a subgraph with given density as $\delta\rightarrow 2$, making these bounds meaningful for every $\delta \in [1,2)$.
\end{abstract} 

\noindent\emph{Keywords and phrases:} Graph algorithms, random graphs, cliques, dense subgraphs, adaptive algorithms\\
\emph{MSC2020 codes:} 05C80, 05C85, 68Q87, 68W20

%%%%%%%%%%%%%%%%
%%% Document %%%
%%%%%%%%%%%%%%%%

%%%%%%%%%%%%%%%%%%%%%%%%%%%%%%%%%%%%%%%%%%%%
\section{Introduction} \label{sec:intro} %%%
%%%%%%%%%%%%%%%%%%%%%%%%%%%%%%%%%%%%%%%%%%%%

Finding a subgraph with certain properties in a graph is a central topic of theoretical computer science. 
It is well-known that finding a maximum clique or a Hamiltonian cycle are NP-complete problems~\cite{Karp1972}. 
Even approximating the size of the maximum clique within a given factor is hard in standard computational models~\cite{FGLSSz91}. 

Ferber et al. proposed the Subgraph Query Problem in \cite{ferber2016finding,ferber2017finding}. 
The general question is to find a subgraph in $G(n,p)$ with high probability that satisfies a given monotone graph property by querying as few pairs as possible, where a query is simply checking whether the given pair constitutes an edge. 
%??? equivalent?
In some sources, the requirement that the algorithm succeeds with high probability is replaced by the equivalent requirement that the algorithm succeeds with probability at least $1/2$. 
In \cite{ferber2016finding,ferber2017finding} Hamiltonian cycles and long paths were considered in sparse Erd\H{o}s-R\'enyi graphs. 
In the same setup, Conlon et al.~\cite{conlon2019online} studied the problem of finding a fixed subgraph (e.g., a clique of given size). 
For the related Planted Clique Problem, see \cite{huleihel2021random,mardia2020finding,RS20_ALEA,rashtchian2021average}.

A natural special case of the Subgraph Query Problem is the Maximum Clique Query Problem (MCQP) introduced in \cite{FGNRT20}: 
for a given $p\in (0,1)$, what is the size of the largest clique that we can find in $G(n,p)$ with high probability by using at most $n^\delta$ queries ($\delta\in [1,2]$)? 
It turns out that the parameter $p$ is not so important (as long as it is a fixed constant): it is usually set to $p=1/2$. 
The present paper also works under this assumption, although the results could be generalized to arbitrary $p\in (0,1)$. 

The size of the largest clique in $G(n,1/2)$ is asymptotically $2\log n$ with high probability, where $\log$ is the base 2 logarithm; for a more precise estimate, see \cite{matula1972,bollobas1976cliques,lugosi17}. 
This classical result answers the question for $\delta=2$: if we are allowed to query all edges, we can find the maximum clique in the graph, and it has approximately $2\log n$ vertices with high probability. 
Note that the complexity of the algorithm is only measured in the number of queries. 
In particular, the actual runtime of the algorithm is irrelevant: even though finding the maximum clique is an NP-complete problem, which means that to our best knowledge it likely requires an exponentially large amount of time in standard computational models, we still view this as a quadratic algorithm in our framework. 

In the $\ell$-adaptive variant of MCQP, the queries are divided into $\ell$ rounds for some $\ell\in \N$. 
The algorithm makes the queries within a round simultaneously. 
The effect of this modification is that within a round, we have to make a decision for several moves ahead, only taking into consideration the results of the previous rounds. 
Informally speaking, we cannot adapt our next choice for the queried pair if another pair in the same round yielded an unfavorable result. 
The original version of the problem can be viewed as the infinitely adaptive variant, that is, $\ell=\infty$. 

The first non-trivial upper bound was shown by Feige at al. in \cite{FGNRT20} for the $\ell$-adaptive version. 
They proved that for all $\delta\in [1,2)$ and $\ell\in \N$ there exists a constant $\alpha<2$ such that it is impossible to find a clique of size $\alpha\log n$ in $\ell$ rounds using $n^\delta$ queries altogether. 
That is, $\alpha_\star(\delta,\ell)<2$ for all $\ell\in \N$. 
Soon after, an improvement was shown in \cite{alweiss2021subgraph} by Alweiss et al. 
They studied the fully adaptive variant ($\ell=\infty$) of MCQP, and proved that $\alpha_\star(\delta, \infty)\leq 1+\sqrt{1-(2-\delta)^2/2}$. 
Clearly, this is also an upper bound for the $\ell$-adaptive variant for any $\ell\in \N$. 
To date, this was the strongest known upper bound, except for the well-understood $\ell=2$ case and some  estimates for $\ell=3$. 
These special cases were investigated in \cite{feige2021tight} by Feige and Ferster. 
In the $\ell=2$ case they have shown that $\alpha_{\star}(\delta, 2) = 4\delta/3$ for $\delta \in [1,6/5]$ and $\alpha_{\star}(\delta, 2) \leq 1 + \sqrt{1-(2-\delta)^{2}}$ for $\delta \in [6/5,2]$. 

In this paper, we improve on the idea introduced by Feige and Ferster in~\cite{feige2021tight}. 
A monotone increasing function $\gamma: \N\cup\{\infty\} \rightarrow \R^+$ is defined in Section~\ref{sec:comb} as a solution of a combinatorial problem. 
Using this function, we re-prove the results of \cite{alweiss2021subgraph} for $\ell=\infty$ and \cite{feige2021tight} for $\ell=2$, and obtain stronger results for $\ell=3$ and non-trivial estimates for all $\ell\geq 4$. 

\begin{theorem}\label{thm:main}
For every $\delta \in [1,2]$ and $\ell \geq 3$, including $\ell=\infty$, we have 
\begin{equation}\label{eq:gammaub}
\alpha_{\star} \left( \delta, \ell \right) 
\leq 
1+\sqrt{1-\frac{(2-\delta)^2}{4\gamma(\ell)}}.
\end{equation}

Furthermore, for $\ell=2$ the same estimate applies for $\delta \in [6/5,2]$, and $\alpha_{\star} \left( \delta, \ell \right)\leq 4\delta/3$ for $\delta \in [1, 6/5]$. 
\end{theorem}

We compute some values of the function $\gamma$ precisely, namely $\gamma(1)=0$, $\gamma(2)=1/4$, $\gamma(3)=3/8$, and $\gamma(\infty)=1/2$.
For $\ell\in\N, \ell\geq 3$ we prove the upper bound $\gamma(\ell)\leq 1/2-1/(3\cdot 2^{\ell-1}-4\ell+8)$; see Theorem~\ref{thm:l-labels}. 
In particular, $\gamma(4)\leq 7/16=0.4375$. 
The bounds provided by Theorem~\ref{thm:l-labels} are probably not tight for $\ell\geq 4$; for $\ell=4$, the best lower bound we could prove is 
%$\gamma(4)\geq (9+\sqrt{3})/26 \approx 0.41277$, 
$\gamma(4)\geq 5/12$, which we also do not believe to be tight.
Nevertheless, putting these values and estimates into Theorem~\ref{thm:main} yields the following more concrete result. 

\begin{corollary}\label{cor:three_rounds_n_queries}
%We have $\alpha_{\star} \left( \delta, 3 \right) \leq 1+\sqrt{1-\frac{2}{3}(2-\delta)^2}$. 
%Moreover, for $\ell\geq 4$ we have 
For every $\delta \in [1,2]$ and $\ell\geq 3$ we have 
\begin{equation}\label{eq:gammaubeta}
\alpha_{\star} \left( \delta, \ell \right) 
\leq 
1+\sqrt{1-\frac{(2-\delta)^2}{2-1/(3\cdot 2^{\ell-3}-\ell+2)}}.
\end{equation}
\end{corollary}

In particular, $\alpha_{\star} \left( \delta, 3 \right) \leq 1+\sqrt{1-\frac{2}{3}(2-\delta)^2}$. 
For instance, $\alpha_{\star} \left( 1, 3 \right) \leq 1+1/\sqrt{3}\approx 1.577$ by Corollary~\ref{cor:three_rounds_n_queries}. 
The earlier best estimate, proven in~\cite{feige2021tight}, was $\alpha_{\star} \left( 1, 3 \right) \leq 1.62$. 

We define the Maximum Dense Subgraph Query Problem (MDSQP), a natural generalization of MCQP. 
Given a $\delta \in [1,2]$, $\ell\in \N\cup\{\infty\}$, and $\eta\in (1/2,1]$. 
The problem is to find the largest possible subgraph by using at most $n^\delta$ queries (and unlimited computational time) in $G(n,1/2)$ with edge density at least $\eta$, with high probability. 
A recent result of Balister et al.~\cite{balister2019dense} determines the size of the largest subgraph in $G(n,1/2)$ having edge density at least $\eta$ (with high probability). 
For $\eta\in (1/2,1]$, it is asymptotically $\frac{2}{1-H(\eta)}\log n$, where $H(\eta)=-\eta\log\eta - (1-\eta)\log(1-\eta)$ is the Shannon entropy. 
Note that this is consistent with the above discussion: cliques in a graph are exactly the subgraphs with edge density $\eta=1$, and indeed $\frac{2}{1-H(1)}\log n = 2\log n$. %, since $H(1)=0$. 
Just as in the Maximum Clique Query Problem, we only expect to achieve this size $\frac{2}{1-H(\eta)}\log n$ by an $\ell$-adaptive algorithm using $n^\delta$ queries if $\delta=2$, that is, when the whole graph is uncovered. 
The natural problem is to determine $\alpha_{\star} \left( \delta, \ell, \eta \right)$, the supremum of all $\alpha$ such that an appropriate $\ell$-adaptive algorithm using $n^\delta$ queries finds a subgraph in $G(n,1/2)$ with density at least $\eta$ and size $\alpha\log n$. 
For a lower bound, a greedy algorithm using a linear number of queries (i.e., $\delta=1$) was presented in \cite{DDK10} by Das Sarma et al. 
Their method could provide a general lower estimate, however it is hard to make it explicit. 
So rather than proving a formula, they focused on a numerical result, and showed that $\alpha_{\star} \left( 1, \infty, 0.951 \right)\geq 2$. 
In other words, there is a fully adaptive algorithm using a linear number of queries that finds a subgraph in $G(n,1/2)$ with size at least $2\log n$ and density at least $0.951$ with high probability. 

Using the same techniques as above for the MCQP, we prove the following upper estimate for $\alpha_{\star} \left( \delta, \ell, \eta \right)$. 

\begin{theorem}\label{thm:dense}
Let $\delta \in [1,2]$, $\ell \geq 3$, including $\ell=\infty$, and $\eta\in (3/4,1]$. 
Given an $\alpha$, we define $m_1$ as the smallest solution of the equation $4\gamma(\ell)m\left(1+\log\frac{\eta\alpha^2/2-2\gamma(\ell)m^2}{\alpha^2/2-2\gamma(\ell)m^2}\right)=2-\delta$ on $[0,\alpha/2]$; if there is no such solution, then $m_1=\alpha_1=\infty$. Let $m_2=\alpha/2$. 
Let $\alpha_i$ be the largest solution of the equation 
$$(\alpha^2/2-2\gamma(\ell)m_i^2)\left(1-H\left(\frac{\eta\alpha^2/2-2\gamma(\ell)m_i^2}{\alpha^2/2-2\gamma(\ell)m_i^2}\right)\right)-\alpha+(2-\delta)m_i= 0$$
%with $m=m_1$ and $m=m_2$, respectively. 
for $i=1,2$. 
Let $\alpha_0=\min(\alpha_1,\alpha_2)$.   
Then $\alpha_{\star} \left( \delta, \ell, \eta \right) \leq \alpha_0$.
\end{theorem}

In contrast to Theorem~\ref{thm:main}, this upper bound is only given implicitly. 
Even telling when $\alpha_0=\alpha_2$ holds, either because $m_1=\infty$ or because $\alpha_1>\alpha_2$, seems to be challenging. 
In MCQP, the analogous degenerate condition only applies when $\ell=2$ and $\delta\in[1,6/5]$, yielding the exceptional case in Theorem~\ref{thm:main}. 
Nevertheless, this formula can be used to obtain numerical results up to any prescribed precision, in principle. 
For instance, this shows that a linear, fully adaptive algorithm ($\delta=1, \ell=\infty$) cannot find a subgraph of size $2\log n$ whose density is at least $0.98226$ with high probability. 
Or, to complement the above lower estimate, it also shows that $\alpha_{\star} \left( 1, \infty, 0.951 \right)\leq 2.48227$. 
The trivial upper bound for $\alpha_{\star} \left( 1, \infty, 0.951 \right)$ is $\frac{2}{1-H(0.951)} < 2.7861$, since there is no subgraph in $G(n,1/2)$ with density at least $0.951$ and size at least $2.7861\log n$ with high probability. 

Note that Theorem~\ref{thm:dense} is a generalization of Theorem~\ref{thm:main}: if $\eta=1$, then $m_1=\frac{2-\delta}{4\gamma(\ell)}$, provided that this value is less than $\alpha/2$. 
Then the defining equation of $\alpha_1$ reduces to the same equation $\alpha^2/2-2\gamma(\ell)m_1^2-\alpha+(2-\delta)m_1= 0$ that yields the formula in Theorem~\ref{thm:main}. 
Furthermore, as $\delta\rightarrow 2$, we have $m_0\rightarrow 0$, and then $\alpha_0$ tends to the solution of the equation $(\alpha^2/2)(1-H(1-\eta))-\alpha=0$. 
Thus $\alpha_0\rightarrow \frac{2}{1-H(\eta)}$, which is the trivial upper bound. 
Hence, for any $\delta\in [1,2)$, the upper bound provided by Theorem~\ref{thm:dense} is strictly smaller than the size of the largest subgraph with density at least $\eta$ (divided by $\log n$), making it a meaningful estimate. 

%Finally, we mention a line of research that is very similar in essence to the above problems.  that finding cliques and large density subgraphs

%%%%%%%%%%%%%%%%%%%%%%%%%%%%%%%%%%%%%%%%%%%%
\section{Two combinatorial problems} \label{sec:comb} %%%
%%%%%%%%%%%%%%%%%%%%%%%%%%%%%%%%%%%%%%%%%%%%

We pose a question concerning labeled graphs that is closely related to the $\ell$-adaptive Maximum Clique Query Problem and Maximum Dense Subgraph Query Problem: an upper estimate to this question yields an upper estimate to both problems. 
Then we present a simplified variant of this combinatorial question which is easier to handle and can be used to attain lower estimates to the original version. 
In fact, the two variants of the question might be equivalent. 
For certain values of the parameters, we will show that they have the same answer, and this could be true in full generality. 

\begin{question}\label{que:labpri}
Given $\ell\in \N\cup\{\infty\},\ M,N \in \N$ with $N\geq 2M$, a labeling $\lambda: E(K_N)\rightarrow \{1,\ldots, \ell\}$ of the edges of the complete graph $K_N$, and a matching $\mathcal M$ in $K_N$ of size $M$, we say that an edge $uv$ is \emph{critical} if there is an edge $e\in \mathcal M$ covering $u$ or $v$ such that $\lambda(uv)<\lambda(e)$. 
For fixed $\ell, M, N$, labeling $\lambda$ and perfect matching $\mathcal M$, let $\gamma(\ell,M,N,\lambda,\mathcal M)$ be the number of critical edges divided by $\binom{2M}{2}$. % as $\mathcal M$ ranges through all perfect matchings of $K_n$. 
For each $\ell\in \N$, %what is the limit superior of the minimum ratio of critical edges in large graphs assuming an optimal choice for the matching? 
%That is, 
find  
$$\gamma(\ell)= \limsup\limits_{M\rightarrow \infty} \ \max\limits_{N\geq 2M} \  \max\limits_{\lambda}\ \min\limits_{\mathcal M} \gamma(\ell,M,N,\lambda,\mathcal M).$$
\end{question}

The following variant of the same question is the special case of perfect matchings in complete graphs, that is, $N=2M$. 

\begin{question}\label{que:lab}
Given $\ell\in \N\cup\{\infty\},\ M \in \N$, $N=2M$, a labeling $\lambda: E(K_N)\rightarrow \{1,\ldots, \ell\}$ of the edges of the complete graph $K_N$, and a perfect matching $\mathcal M$ in $K_N$, we say that an edge $uv$ is \emph{critical} if $\lambda(uv)$ is strictly less than the maximum of the labels of the two edges in $\mathcal M$ covering $u$ and $v$. 
For fixed $\ell, N$, labeling $\lambda$ and perfect matching $\mathcal M$, let $\gamma'(\ell,N,\lambda,\mathcal M)$ be the ratio of critical edges in the $\binom{N}{2}$ edges. % as $\mathcal M$ ranges through all perfect matchings of $K_n$. 
For each $\ell\in \N$, %what is the limit superior of the minimum ratio of critical edges in large graphs assuming an optimal choice for the matching? 
%That is, 
find  
$$\gamma'(\ell)= \limsup\limits_{N\rightarrow \infty} \  \max\limits_{\lambda}\ \min\limits_{\mathcal M} \gamma(\ell,N,\lambda,\mathcal M).$$
\end{question}

\begin{remark}
Using the language of graph limits~\cite{Lovaszbook}, Question~\ref{que:lab} has the following equivalent reformulation. 
This equivalence also implies that $\limsup$ in Question~\ref{que:lab} can be replaced by $\lim$.

Given $\ell \in \N$, and a measurable labeling $\lambda: [0, 1] \times [0, 1] \rightarrow \{1,\ldots, \ell\}$ of the edges of the complete graphon.
Or for $\ell = \infty$, $\lambda: [0, 1] \times [0, 1] \rightarrow [0, 1]$.
Given a measure-preserving bijection $\mathcal M: [0, 1] \to [0, 1]$, we say that an edge $(u,v) \in [0,1] \times [0,1]$ is \emph{critical} if $\lambda(u,v) < \max\Big(\lambda\big(u, \mathcal M(u)\big),\  \lambda\big(v, \mathcal M(v)\big)\Big)$.
For fixed $\ell$, measurable labeling $\lambda$, and measure-preserving bijection $\mathcal M$, let $\gamma'(\ell,\lambda,\mathcal M)$ be the measure of critical edges. % as $\mathcal M$ ranges through all measure-preserving bijections. 
For each $\ell\in \N$, %what is the maximum measure of critical edges in the complete graphon assuming an optimal choice for the matching $\mathcal M$? 
%That is, 
find 
$$\gamma'(\ell)= \max\limits_{\lambda} \ \min\limits_{\mathcal{M}} \gamma(\ell,\lambda,\mathcal{M}).$$ 

%??? For $\ell = \infty$, we should define $\lambda: [0, 1] \times [0, 1] \rightarrow [0, 1]$.
\end{remark}

%The connection to the adaptive clique problem is, roughly speaking, $\alpha_{\star} \left( \delta, \ell \right) 
%\leq 
%1 + \sqrt{1 - \frac{(2-\delta)^{2}}{4\gamma(\ell)}}$, see Theorem~\ref{thm:main} for the precise statement. 
%Moreover, the above problem, or a meaningful upper bound to it, is just as useful in the dense subgraph problem; see ???. 

Obviously, $\gamma'(\ell)\leq \gamma(\ell)$ for all $\ell\in\mathbb{N}$; furthermore, $0=\gamma'(1)\leq \gamma'(2)\leq \cdots\leq \gamma'(\infty)$, $0=\gamma(1)\leq \gamma(2)\leq \cdots\leq \gamma(\infty)$. %, where $\gamma(\infty)$ is the natural generalization of the problem such that any positive integer can be used as a label. 
If $\ell=\infty$, it is not worth using the same label twice (in a finite complete graph), hence the problem can be rephrased as follows. 
%(We phrase the simpler one.)
Consider all $\binom{N}{2}!$ orders of the edges of $K_N$. 
Given a matching $\mathcal M$, an edge $uv$ is critical if an edge in $\mathcal M$ covering $u$ or $v$ appears later in the ordering than $uv$. 
Then $\gamma(\infty,n,\lambda,\mathcal M)$ is the ratio of critical edges, and $\gamma(\infty)= \limsup\limits_{n\rightarrow \infty} \  \max\limits_{\lambda}\ \min\limits_{\mathcal M} \gamma(\infty,n,\lambda,\mathcal M)$. 

Given a matching $\mathcal M$ in $K_N$, we call $b$ the \emph{$\mathcal M$-neighbor} of $a$ if $ab\in \mathcal M$. 
Moreover, the \emph{$\mathcal M$-pair} of an edge $e=uv\notin \mathcal M$ is the edge linking the $\mathcal M$-neighbor of $u$ and the $\mathcal M$-neighbor of $v$. 
Note that this is a proper pairing of non-matching edges that link vertices covered by matching edges. 
(In particular, if $\mathcal M$ is a perfect matching, than it is a proper pairing of non-matching edges in the whole graph $K_N$.)
The \emph{$e$-switch} of $\mathcal M$ is the operation that replaces by $e$ and its $\mathcal M$-pair $e'$ the two edges in $\mathcal M$ that cover the same quadruple of vertices as $e$ and $e'$. 
This produces a new matching of $K_N$ of the same size as $\mathcal M$. 

An \emph{outward critical edge} is a critical edge $uv$ such that exactly one of the endpoints $u$ or $v$ is covered by a matching edge in $\mathcal M$. 
Note that given an $\ell\in \N\cup\{\infty\}$, if an optimal construction for a matching to Question~\ref{que:labpri} (that is, one that attains the critical edge ratio $\gamma'(\ell)$ asymptotically) has no outward critical edge, then it is also an optimal construction for a matching to Question~\ref{que:lab}, and then $\gamma'(\ell)=\gamma(\ell)$. 
More generally speaking, if a construction for a matching has no outward critical edge, then the limiting critical edge ratio obtained from this construction is an upper estimate to both $\gamma'(\ell)$ and $\gamma(\ell)$. 
All constructions in the present paper has this property. 
In fact, we believe that $\gamma'(\ell)=\gamma(\ell)$ for all $\ell\in \N\cup\{\infty\}$. 

We first solve the $\ell=\infty$ case by using a similar idea as that in the proof of \cite[Lemma 13]{alweiss2021subgraph}, except we need to construct the matching in a more complicated way. 

\begin{proposition}\label{prop:linfty}
$\gamma'(\infty)=\gamma(\infty)=1/2$
\end{proposition}
\begin{proof}
%Given an ordering of the set of edges of a graph $G$, let $G_{\prec e}$ be the initial segment of the total order up to the edge $e$, not including $e$. 
Given an ordering $\prec$ of the set of edges of a complete graph $K_N$ with $N\geq 2M$, let $\mathcal M$ be the minimal matching of size $M$ in the anti-lexicographic order. 
%We construct a matching $\mathcal M$ of size $M$ with edges $m_1\prec m_2\prec \cdots \prec m_{M}$ such that for every $1\leq k\leq M$ the graph
%$K_N[m_1, m_2, ..., m_k]_{\prec m_k}$ has no perfect matching of size $k$, 
%where $G[m_1, m_2, ..., m_k]$ denotes the clique of $G$ spanned by the listed edges. 
%We can construct $\mathcal M$ recursively in decreasing order of $k$.
%Namely, we delete the edges of $K_N$ in decreasing order, and whenever the size of the maximum matching decreases, we add that edge to the matching and delete its two endpoints from the graph.
%In other words, $\mathcal M$ is minimal in the anti-lexicographic order of matchings of size $M$. 

Assume that a non-matching edge $e$ and its $\mathcal M$-pair $e'$ are both critical. 
Then the matching obtained from $\mathcal M$ as the result of the $e$-switch is smaller in the anti-lexicographic order than $\mathcal M$, a contradiction. 
Moreover, there cannot exist an outward critical edge $e$, as we could switch the matching edge $m\in \mathcal M$ sharing an endpoint with $e$ to the edge $e$, obtaining a new matching that is smaller than $\mathcal M$ in the anti-lexicographic order. 
%Let $m_i\prec m_k$ be the two matching edges in $\mathcal M$ that cover the same four points as $e$ and $e'$. 
%Then $K_n[m_1, m_2, ..., m_k]_{\prec m_k}$ has a perfect matching, obtained as the result of the $e$-switch restricted to $K_n[m_1, m_2, ..., m_k]$, a contradiction. 
Hence, the number of critical edges is at most $\big(\binom{2M}{2}-M\big)/2$, yielding $\gamma(\infty)\leq 1/2$. 

For the lower bound $\gamma'(\infty)\geq 1/2$, enumerate the $N=2M$ vertices of $K_N$, and let $\lambda$ be the lexicographical order of the edges: 
$12, 13, 14, \ldots, 1N, 23, 24, \ldots, 2N, \ldots, (N-1)N$. 
Let $\mathcal M$ be any perfect matching. 
If $uv\in \mathcal M$ for some $u<v$, then it makes exactly the edges $iv$ and $uj$ critical for all $1\leq i\leq u-1$ and $1\leq j\leq v-1$, $j\neq u$. 
That is, the edge $uv\in \mathcal M$ makes exactly $u+v-3$ edges critical. 
As each number between 1 and $N$ appears exactly once as an endpoint in a matching edge, the sum of all these expressions $u+v-3$ for edges $uv\in \mathcal M$ is $\frac{N(N+1)}{2}-\frac{3N}{2}$, which is asymptotically the number of edges. 
This is the number of critical edges with multiplicity: each critical edge $e$ contributes one or two into this sum, depending on whether the label $\lambda(e)$ is less than only one or both labels of edges in $\mathcal M$ covering the endpoints of $e$. 
As all multiplicities are at most two and they add up to (approximately) the number of edges, (approximately) at least half of the edges must be critical. 
\end{proof}

We note that the matching $1N,2(N-1),\ldots,(N/2)(N/2+1)$ produces approximately $N^2/4$ critical edges (about half the number of edges) in the construction where edges are labeled in lexicographical order. 
This is exactly the matching defined in the first half of the proof of Proposition~\ref{prop:linfty}. 

Proposition~\ref{prop:linfty} is already strong enough to reprove the upper bound on the fully adaptive clique problem that was shown in \cite{alweiss2021subgraph}. 
Moreover, it provides the upper bound $\gamma(\ell)\leq 1/2$ for all $\ell\in \mathbb{N}$. 
%The best known upper estimate for the $\ell$-adaptive clique problem so far was the same estimate, except for the $\ell=2$ case, which is well-understood. 
Now we improve on this estimate: this is going to yield a better upper bound for the $\ell$-adaptive MCQP for $\ell\geq 3$ than the state of the art (and reproves the best known estimates for $\ell=2$). 

\begin{theorem}\label{thm:l-labels}
\quad
%The asymptotic worst critical edge ratio in large graphs for $\ell=2$ and $\ell=3$ are $\gamma(2)=1/4$ and $\gamma(3)=3/8$. Moreover, for $\ell\geq 4$ we have $\gamma(\ell)\leq 1/2-1/(3\cdot 2^{\ell-1}-4\ell+8)$. 
%For $\ell=2$ and $\ell=3$ we have $\gamma(2)=1/4$ and $\gamma(3)=3/8$. Moreover, for $\ell\geq 4$ we have $\gamma(\ell)\leq 1/2-1/(3\cdot 2^{\ell-1}-4\ell+8)$.
\begin{itemize}
\item $\gamma'(2)=\gamma(2)=1/4$
\item $\gamma'(3)=\gamma(3)=3/8$
\item $\gamma'(\ell)\leq \gamma(\ell)\leq 1/2-1/(3\cdot 2^{\ell-1}-4\ell+8)$ for $\ell\geq 4$
\end{itemize}
\end{theorem}

\noindent The rough idea of the proof of Theorem~\ref{thm:l-labels} is to find a matching such that 
\begin{enumerate}
\item there are no outward critical edges, 
\item at most half of those edges are critical that link matching edges with different labels, and
\item significantly less than half of those edges are critical which link matching edges of the same label.
\end{enumerate}
It is not surprising that if all three goals are fulfilled, then the critical edge ratio is pushed below $1/2$ by some fixed constant (depending on $\ell$).
The next lemma is the crucial tool to achieve the third goal. 

\begin{lemma}\label{lem:alter}
Let $k\in \N$ be a fixed number. 
%Given an $x\in \N$ and a graph with $2x$ vertices and $x$ disjoint red edges. 
For an $x \in \N$, consider the graph on $2x$ vertices with $x$ disjoint edges, colored by red.
Let $\beta_k(x)$ be the largest number of blue edges that can be added to the red perfect matching in the graph so that 
\begin{itemize}
    \item there are no alternating cycles, and 
    \item there are no alternating paths containing at least $k$ blue edges. 
\end{itemize}
Then $\beta_k(x)= (1-1/k)x^2+O_k(x)$ as $x\rightarrow \infty$. 
\end{lemma}

\begin{proof}
For the lower estimate, we show two different constructions: one for even and one for odd $k$. 
If $k$ is even, let us partition the set of red edges into $k/2$ subsets of roughly equal size: that is, the cardinality of any two should differ by at most 1. 
The possible differences between these sets only yield an $O(x)$ error, so we do the calculation assuming that each set contains $2x/k$ red edges. 
The set of vertices covered by these sets of edges are $X_1, \ldots, X_{k/2}$, each containing $4x/k$ vertices. 
From each pair of vertices that are red neighbors, we pick one and call it the left vertex of the edge; the other one is the right vertex of the edge. 
Thus in $X_i$ there are $2x/k$ left vertices and $2x/k$ right vertices. 
The left vertices are all linked by blue edges, contributing $\binom{x}{2}=x^2/2+O(x)$ blue edges. 
Two right vertices are never blue neighbors. 
A right vertex $u$ is linked to a left vertex $v$ iff the index of the set containing $u$ is larger than that of $v$. 
See the left picture in Figure~\ref{fig:evenk} for an illustration. 

There are $\binom{k/2}{2}=k(k-2)/8$ pairs $(i,j)$ with $1\leq i<j\leq k/2$, and for all such pairs there are $4x^2/k^2$ left-right blue edges between $X_i$ and $X_j$, contributing $(1/2-1/k)x^2$ blue edges. 
Thus there are $(1-1/k)x^2+O(x)$ blue edges altogether. 

\begin{figure}
\begin{center}
\includegraphics[scale=0.5]{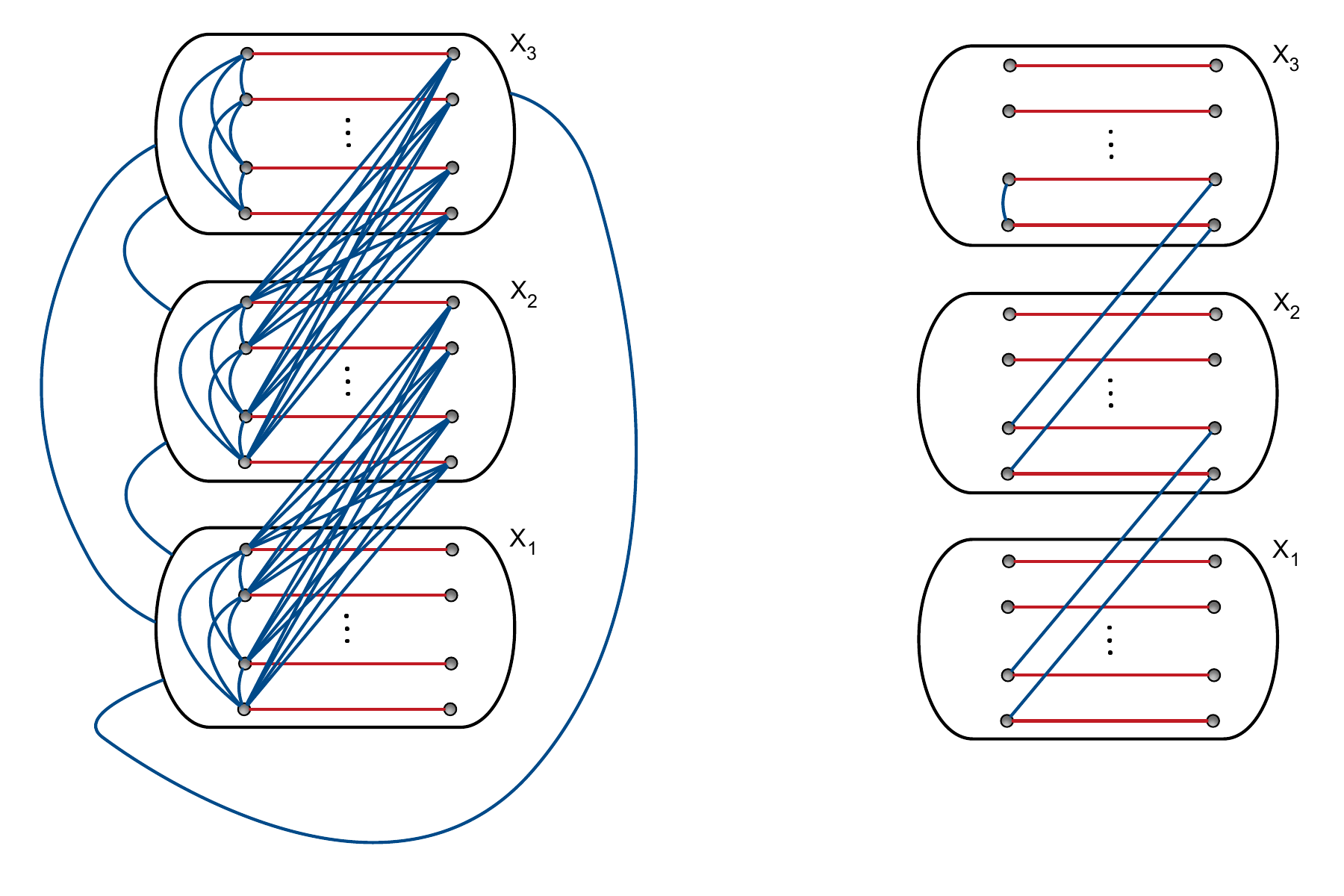}
\end{center}
\caption{Left: construction for k=6. Right: an alternating path containing the most blue edges.}
\end{figure}\label{fig:evenk}

Assuming there is an alternating cycle in this graph, let us pick a red edge in that cycle, and consider its right endpoint. 
The red edge must be followed by a blue one, and right vertices are only linked to left vertices in the graph by blue edges. 
Moreover, this blue neighbor of the right vertex is in a lower-index set. 
Thus, the next vertex in the cycle must be on the left, and in a lower-index set. 
Then we have to follow up by the only available red edge, which means that we move to the right side once again. 
Hence, vertices in any alternating cycle alternate between the left and right side, and the right vertices along the cycle are in sets of descending index. 
Such a descending walk cannot be circular, a contradiction.

The longest alternating path containing the largest number of blue edges starts off at the right of $X_1$, followed by the red pair of this vertex, then crosses to the right side of $X_2$, followed by the red pair of that vertex, etc. 
When we enter the right side of $X_{k/2}$, we move to the pair of that vertex. 
Then we can pick any other vertex at the left side of $X_{k/2}$, as there are blue edges between left vertices. 
Then we have to move to the red neighbor of the last vertex, and walk backwards in a similar zig-zag fashion until arriving at the right side of $X_1$. 
There are $2(k/2-1)$ cross edges between left and right in such a path, and one more blue edge in $X_{k/2}$, which is $k-1$ blue edges altogether. 
See the right picture in Figure~1 for an illustration. 

The construction for odd $k$ is somewhat more roundabout. 
This time there are $(k+1)/2$ sets $X_1, \ldots, X_{(k-1)/2}, X_{(k+1)/2}$, and the last one is half as big as the rest. That is, the number of vertices in $X_i$ is $2x_i$, where $x_1=\cdots=x_{(k-1)/2}=2x/k$ and $x_{(k+1)/2}=x/k$. 
Otherwise, the construction is the same, except that there are no blue edges linking two left vertices of $X_{(k+1)/2}$. 
The argument that this graph contains no alternating cycle is the same as before. 
A longest path can still make its way from $X_1$ up to $X_{(k+1)/2}$ in a zig-zag motion, but it must turn back immediately without gaining an edge on the left side in $X_{(k+1)/2}$, as there are no blue edges linking left vertices of $X_{(k+1)/2}$. 
That is, when we reach a left vertex in $X_{(k+1)/2}$, the best we can do is to drop down to the left of $X_{(k-1)/2}$, and zig-zag all the way down to $X_1$. 
Hence, the most blue edges in an alternating path is $2(k-1)/2=k-1$. 
The union of the first $X_{(k-1)/2}$ sets is the same as the construction for the even number $k-1$ on $(k-1)x/k$ red edges, thus there are $(1-1/(k-1))((k-1)x/k)^2+O(x)=(k-2)(k-1)x^2/k^2+O(x)$ blue edges in that induced subgraph. 
In addition, all $2x/k$ vertices in $X_{(k+1)/2}$ have $(k-1)x/k$ blue neighbors, contributing $2(k-1)x^2/k^2$ blue edges. 
Thus there are $((k-2)(k-1)+2(k-1))x^2/k^2+O(x)=(1-1/k)x^2+O(x)$ blue edges in this graph. 

We prove the upper estimate by induction on $k$. 
Clearly, $\beta_1(x)=0$, which is consistent with the formula for $k=1$. 
If $k=2$, then there cannot be any red edges both of whose endpoints have blue degree at least 2. 
Indeed, if there were such a red edge, then we could match the two endpoints with different blue neighbors, yielding an alternating path with two blue edges. 
Thus there are $O(x)$ blue edges incident to vertices with blue degree at most 1, and every red edge contains such a vertex. 
At worst, all other vertices are linked by blue edges, which yields $\binom{x}{2}+O(x)=x^2/2+O(x)$ blue edges altogether, consistently with the formula for $k=2$. 
Let $k\geq 3$ and assume that the assertion holds for all smaller values of $k$. %(actually, for $k-1$ and $k-2$)

Let $G$ be a graph with maximum number $\beta_k(x)$ of blue edges satisfying the requirements. 
If an alternating path in $G$ ends in a blue edge, we can always extend it by the red edge incident to its last vertex. 
The only obstruction to the addition of this edge would be if the other endpoint of the red edge coincided with the starting vertex of the path. 
However, that would yield an alternating cycle. 
Hence, alternating paths with the most blue edges in them are exactly the longest alternating paths in $G$ (after adding red edges at the end, if necessary). 
Let $P$ be a longest alternating path in $G$. 
We may assume that $P$ contains $k-1$ blue edges, otherwise the formula for $\beta_{k-1}(x)$ would apply, yielding at least as strong an upper estimate as the one claimed for $\beta_k(x)$. 
In particular, there are $k$ red edges in $P$. 

Each endpoint $u$ of $P$ has all blue neighbors in $P$, as otherwise the path could be extended by a blue edge. 
The red neighbor of $u$, that is, the second vertex in the path starting from $u$, cannot be a blue neighbor of $u$. 
In any other red edge contained in the path there is a vertex $v$ such that if $uv$ were a blue edge, then it would form an alternating cycle together with the segment of $P$ from $u$ to $v$. 
Hence, each endpoint of $P$ has blue degree at most $k-1$ in $G$. 

Let $D$ be the set of vertices in $G$ that are incident to a red edge $e$ which has at least one endpoint of blue degree at most $k-1$. 
Let $d$ be the number of red edges in $D$. 
Delete the $2d$ vertices of $D$ from $G$ to obtain the graph $G'$. 
Let $P'$ be a longest path in $G'$. 
By repeating the same argument as above, $G'$ starts and ends in a red edge. 
As we deleted all vertices from $G$ that could be endpoints of longest paths, $P'$ contains at most $k-2$ blue edges. 

\begin{figure}
\begin{center}
\includegraphics[scale=0.55]{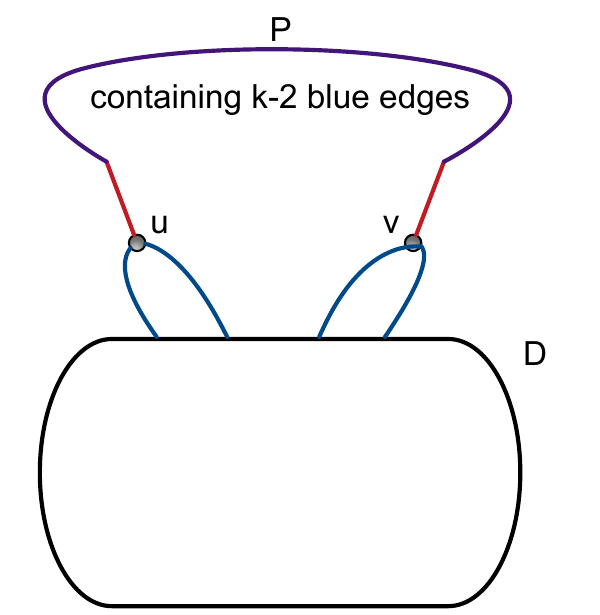}
\end{center}
\caption{A path with many blue edges after deleting $D$.}
\end{figure}\label{fig:longestpath} 

Seeking for a contradiction, assume that $P'$ contains exactly $k-2$ blue edges. 
We can repeat the above argument to show that any of the two endpoints $u,v$ of $P'$ has blue degree at most $k-2$ in $G'$. 
As $u,v\notin D$, they both have blue degree at least $k$ in $G$. 
Thus both $u$ and $v$ are linked to at least two points in the deleted set $D$ of vertices by blue edges of $G$; see Figure~2. 
Let $w\in D$ be such a vertex in $D$ linked to $u$ by a blue edge. 
%The red neighbor of $w$ cannot be linked to $v$ by a blue edge, as that would yield an alternating cycle. 
Out of the at least two blue neighbors of $v$ in $D$, there must be at least one vertex $z\neq w$. %both $w$ and its red neighbor. 
Then we can extend the path $P'$ by the blue edges $uw$ and $vz$ to obtain an alternating path with $k$ blue edges in $G$, a contradiction. 

Hence, there is no alternating path in $G'$ containing at least $k-2$ blue edges. 
Clearly, there is also no alternating cycle in $G'$, as that would be an alternating cycle in $G$. 
Thus the conditions of the lemma apply to $G'$ with fixed constant $k-2$. 
Therefore, there are at most $\beta_{k-2}(x-d)$ blue edges in $G'$. 
One endpoint of every red edge in $D$ contributes at most $k-1$ further blue edges; this is at most $d(k-1)$ blue edges altogether, which we are simply going to estimate by $kx$ from above. 
Not counting these blue edges again, the remaining $d$ points in $D$ can only be linked to each other and to the $2x-2d$ vertices in $G'$, contributing at most $\binom{d}{2}+2d(x-d)\leq 2dx-\frac{3}{2}d^2$ further blue edges. 
Hence, 
$$\beta_k(x)\leq \beta_{k-2}(x-d) + 2xd - \frac{3}{2}d^2 +kx$$
for some $0\leq d\leq x$. 
By the induction hypothesis, there is a $c_{k-2}\in \R$ such that $\beta_{k-2}(x-d)\leq (1-1/(k-2))(x-d)^2+c_{k-2}(x-d)$. 
Thus 
$$\beta_k(x)\leq (1-1/(k-2))(x-d)^2+c_{k-2}(x-d) + 2xd - \frac{3}{2}d^2 +kx\leq$$
$$\frac{k-3}{k-2}(x-d)^2 + 2xd - \frac{3}{2}d^2 + (c_{k-2}+k)x$$
for some $0\leq d\leq x$. 
The derivative of this quadratic function with respect to the variable $d$ is $\frac{2k-6}{k-2}(d-x)+2x-3d=\frac{2}{k-2}x-\frac{k}{k-2}d$, thus the maximum is attained at $d=2x/k$; cf. the constructions for the lower bound. 
By substituting $d=2x/k$ into the expression, we obtain the upper bound  
$$\beta_k(x)\leq \frac{k-3}{k-2}\frac{(k-2)^2}{k^2}x^2+4x^2/k-6x^2/k^2+O_k(x)=$$
$$\left(\frac{k^2-5k+6}{k^2}+\frac{4k}{k^2}-\frac{6}{k^2}\right)x^2+O_k(x)=(1-1/k)x^2+O_k(x).$$
\end{proof}

We need a final technical lemma before proving the main result of this section, Theorem~\ref{thm:l-labels}.  

\begin{lemma}\label{lem:weights}
%Let $X_1\cup X_2\cup \cdots \cup X_\ell \subseteq K_N$ be a partition of a 2M-element subset of the vertex set of the labeled complete graph $K_N$ with $\ell$ labels so that in a matching $\mathcal M$ with size $M$ the matching edges with label $t$ are in $X_t$. 
Let $K_N$ be a labeled complete graph with $\ell$ labels, and let $\mathcal M$ be a matching of size $M$. 
Denote by $X_t$ the vertices covered by matching edges with label $t$. 
Let $|X_t|=2x_t$. 
Assume that there are no outward critical edges, at most half of the edges are critical between different $X_i$ and $X_j$, and within each $X_t$ there are at most $(1-1/k_t)x_t^2+O_\ell(x_t)$ critical edges for some $k_t\geq 1$. 
Let $S=\sum\limits_{t=1}^\ell k_t$. 
Then the number of critical edges is at most $\left(\frac{1}{2}-\frac{1}{2S}\right)\binom{2M}{2}+O_\ell(M)$.
\end{lemma}
\begin{proof}
The errors add up to $O_\ell(n)$, so we disregard them. 
We need to solve the following conditional optimization problem: 
Under the conditions $0\leq x_t$ for all $t$ and $x_1+x_2+\cdots +x_\ell=M$, find the maximum of 
$$\sum\limits_{1\leq i<j\leq \ell} 2x_ix_j + \sum\limits_{t=1}^{\ell} (1-1/k_t)x_t^2.$$
Note that $\sum\limits_{1\leq i<j\leq \ell} 2x_ix_j= \sum\limits_{i\neq j} x_ix_j = \sum\limits_{t=1}^{\ell} x_t(M-x_t)= M\cdot\sum\limits_{t=1}^{\ell} x_t -\sum\limits_{t=1}^{\ell} x_t^2= M^2-\sum\limits_{t=1}^{\ell} x_t^2$. 
This observation leads to a simplified equivalent formulation of our task: 
Under the conditions $0\leq x_t$ for all $t$ and $x_1+x_2+\cdots +x_\ell=M$, find the maximum of $$M^2- \sum\limits_{t=1}^{\ell} x_t^2/k_t.$$
An application of Lagrange multipliers shows that the maximum is attained at $x_t=M k_t/S$. 
This yields the optimum 
$$M^2-\sum\limits_{t=1}^{\ell} \big(M^2k_t^2/S^2\big)/k_t = M^2\bigg(1 -\sum\limits_{t=1}^{\ell} k_t/S^2\bigg)= M^2(1 - 1/S) = \left(\frac{1}{2}-\frac{1}{2S}\right)\binom{2M}{2} + O(M).$$
%As $n^2/2$ is approximately the number of edges in the complete graph $K_n$, this estimate translates to $\gamma(\ell)\leq 1/2 - 1/(2S_\ell)$. 
\end{proof}

Lemma~\ref{lem:alter} and Lemma~\ref{lem:weights} together outline the following strategy to estimate $\gamma(\ell)$ from above. 
Assume that given any $\ell$-labeling of $K_N$, we can find a (red) matching $\mathcal{M}$ of size $M$ with no outward critical (blue) edges such that at most half of the edges between different $X_i$ and $X_j$ are critical, and within an $X_t$, there is no alternating cycle and no alternating path with $k_t$ blue edges. 
Then $\gamma(\ell)\leq \frac{1}{2}-\frac{1}{2S}$, where $S=\sum\limits_{t=1}^{\ell} k_t$. 
The idea is similar to that in the proof of Proposition~\ref{prop:linfty}. 
The matching $\mathcal{M}$ is defined as a minimal element of a carefully chosen (quasi-)order. 
If this order is defined properly, then it is immediate that there is no outward critical edge and no $\mathcal{M}$-pair of critical edges $e,e'$, as the obvious switching operation would yield a smaller element in the order. 
The novelty to the proof of Proposition~\ref{prop:linfty} is the treatment of critical edges within each $X_t$, so that we can apply Lemma~\ref{lem:alter} and Lemma~\ref{lem:weights}. 
To this end, a more complicated order is defined than that in the proof of Proposition~\ref{prop:linfty}. 

\begin{proof}[Proof of Theorem~\ref{thm:l-labels}]
Let $\ell, M,N\in \mathbb{N}$ be such that $2M\leq N$. 
Let $\lambda: E(K_N)\rightarrow \{1,2,\ldots,\ell\}$ be a labeling of the edges of a complete 
graph on $N$ vertices. 
We assign weights to the edges, depending on their label. 
Label 1 edges have weight 0, and label 2 edges have weight 1. 
Then for a small $\eps>0$, the weight assigned to label 3 edges is $2+\eps$, to label 4 edges it is $4+2\eps+\eps^2$, etc. 
In general, 
\begin{itemize}
\item label 1 edges have weight 0, and 
\item for $t\geq 2$, label $t$ edges have weight $\sum\limits_{s=0}^{t-2} 2^{t-2-s}\eps^s$. 
\end{itemize}
Let $\mathcal M$ be a matching of size $M$ with minimum total weight. 
Let $X_t$ be the set of vertices covered by the $x_t$ edges of label $t$ in $\mathcal M$. 
Then putting $|X_t|=2x_t$ we have $x_1+x_2+\cdots +x_\ell=M$. 

We use the terminology introduced before Proposition~\ref{prop:linfty}. 
There is no outward critical edge $e$ (incident with a matching edge $m\in \mathcal M$), as switching $m$ to $e$ would yield a matching with smaller total weight. 
There is no $\mathcal M$-pair of critical edges $e,e'$ between different $X_t$. 
Indeed, the weights are non-negative, monotone increasing, and the weight assigned to each label is more than twice the weight assigned to the previous one. 
Hence, the sum of the weights of $e$ and $e'$ would be strictly less than that of the two matching edges covering the same quadruple of vertices, and then the $e$-switch would decrease the total weight. 
Thus at most half of the edges running between different $X_t$ are critical. 

We now estimate the number of critical edges in each $X_t$. 
For all $\ell\geq 2$, we define a vector $c_\ell$ of length $\ell$. 
For small values of $\ell$ these vectors are $c_2=(1,1)$, $c_3=(1,2,1)$, $c_4=(1,4,2,1)$, $c_5=(1,8,6,2,1)$, $c_6=(1,16,14,6,2,1)$. 
The precise definition is 
\begin{itemize}
\item $c_2=(1,1)$, and
\item for all $\ell\geq 3$, $c_\ell[1]=c_\ell[\ell]=1$, $c_\ell[2]=2^{\ell-2}$, and for all $3\leq t \leq \ell-1$ we have $c_\ell[t]=2^{\ell-t+1}-2$. 
\end{itemize}
Our goal is to show that the number of critical edges in $X_t$ is at most $(1-1/c_\ell[t])x_t^2+O_{\ell}(x_t)$, so that we can apply Lemma~\ref{lem:weights}. 
According to Lemma~\ref{lem:alter}, it is enough to show that if we restrict the matching $\mathcal M$ to $X_t$ (red edges), and color edges of label less than $t$ blue, then there is no alternating cycle and there is no alternating path with $c_\ell[t]$ blue edges in this red and blue subgraph with vertex set $X_t$. 

Clearly, there is no alternating cycle in this subgraph: by switching the red edges of that cycle in $\mathcal M$ to the blue edges of that cycle, we would decrease the total weight of the perfect matching. 
For $t=1$, there cannot be a critical edge in $X_t$ because there is no label less than 1. 
For $t=\ell$, again, there cannot be a critical edge $e$ in $X_t$ because the $e$-switch would decrease the total weight. 
Hence, $c_\ell[1]=c_\ell[\ell]=1$ is justified: there is no alternating path with $1$ blue edge either in $X_1$ or in $X_\ell$. 
For $t=2$, an alternating path with $2^{\ell-2}$ blue edges would have $2^{\ell-2}+1$ red edges. 
The total weight of red edges in such a path is $2^{\ell-2}+1$. 
We propose to switch these red edges in $\mathcal M$ to the blue ones together with the edge linking the endpoints of the path. 
At worst, the endpoints are linked by a label $\ell$ edge. 
As all blue edges have label 1, and consequently weight 0, the total weight of these $2^{\ell-2}+1$ edges is the weight of the label $\ell$ edge, that is, $\sum\limits_{s=0}^{\ell-2} 2^{\ell-2-s}\eps^s$. 
If $\eps$ is small enough, then $\sum\limits_{s=0}^{\ell-2} 2^{\ell-2-s}\eps^s < 2^{\ell-2}+1$. 
Hence, the switch along this cycle (the path together with the edge linking the ends) decreases the total weight of the matching. 
Thus there cannot be an alternating path with $2^{\ell-2}$ blue edges in $X_2$, justifying the formula $c_\ell[2]=2^{\ell-2}$. 
Finally, for $3\leq t\leq \ell-1$, we proceed in a similar fashion: assume that there is an alternating path with $c_\ell[t]=2^{\ell-t+1}-2$ blue edges in $X_t$. %, where the red edges in $X_t$ are the matching edges (of label $t$), and the blue edges in $X_t$ are the critical ones, i.e., edges of label less than $t$ in $X_t$. 
Such a path contains $2^{\ell-t+1}-1$ red edges and $2^{\ell-t+1}-2$ blue edges. 
The total weight of the red edges is 
$$(2^{\ell-t+1}-1)\sum\limits_{s=0}^{t-2} 2^{t-2-s}\eps^s= \sum\limits_{s=0}^{t-2} (2^{\ell-1-s}-2^{t-2-s})\eps^s.$$ 
The total weight of the blue edges together with the edge linking the endpoints is largest if all blue edges have label $t-1$ and the added edge has label $\ell$. 
If this is the case, then the total weight is $$\sum\limits_{s=0}^{\ell-2} 2^{\ell-2-s}\eps^s+(2^{\ell-t+1}-2)\sum\limits_{s=0}^{t-3} 2^{t-3-s}\eps^s = \sum\limits_{s=0}^{t-3} (2^{\ell-1-s}-2^{t-2-s})\eps^s + \sum\limits_{s=t-2}^{\ell-2}2^{\ell-2-s}\eps^s.$$ 
The coefficients of $\eps^s$ in the two sums coincide for $0\leq s\leq t-3$. 
The first difference occurs for $s=t-2$: in the red sum, the coefficient of $\eps^{t-2}$ is $2^{\ell-t+1}-1$, and in the ``blue'' sum it is $2^{\ell-t}$. 
The latter is always smaller than the former as $t\leq \ell-1$. 
Thus for a small enough $\eps$ we could once again improve the total weight of the matching, a contradiction. 
Note that for any fixed $\ell$, only finitely many requirements were made for $\eps$, and all of them hold on an open interval with left endpoint zero and a positive right endpoint. 
Hence, for each $\ell\in \N$ there is a small enough $\eps=\eps(\ell)>0$ that meets all requirements. 

Let $S_\ell=\sum\limits_{t=1}^\ell c_\ell[t]$. 
According to Lemma~\ref{lem:weights} we have $\gamma(\ell)\leq \frac{1}{2}-\frac{1}{2S_\ell}$. 
An elementary calculation yields that $S_2=2$ and that for all $l\geq 3$ we have $S_\ell=3\cdot 2^{\ell-2}-2\ell+4$. 
This translates to the upper bounds $\gamma(2)\leq 1/4$ and $\gamma(\ell)\leq 1/2-1/(3\cdot 2^{\ell-1}-4\ell+8)$ for $\ell\geq 3$. 
In particular, $\gamma(3)\leq 3/8$. 

For the lower estimates $\gamma'(2)\geq 1/4$ and $\gamma'(3)\geq 3/8$, we provide two constructions. 
For $\ell=2$, partition the set of vertices of $K_N$ into two subsets $U_1$ and $U_2$, where $|U_1|=N/4$ and $|U_2|=3N/4$. 
Edges in $U_i$ have label $i$, and edges between the two sets have label 1. 
Given a perfect matching $\mathcal M$, let $Nx$ be the number of edges of $\mathcal M$ between $U_1$ and $U_2$. 
Clearly $x\leq 1/4$, and there are $N(1/8-x/2)$ edges of $\mathcal M$ in $U_1$ and $N(3/8-x/2)$ edges of $\mathcal M$ in $U_2$. 
Critical edges must have label 1, thus they have to lie between $U_1$ and $U_2$ such that the endpoint in $U_2$ is covered by one of the $N(3/8-x/2)$ edges of $\mathcal M$ in $U_2$. 
That is, there are $N(3/4-x)$ possibilities for the endpoint in $U_2$, and $N/4$ possibilities for the endpoint in $U_1$. 
Hence, the number of critical edges is $N^2(3/4-x)/4$, which attains its minimum $N^2/8$ at $x=1/4$. 

For $\ell=3$, let us partition the vertex set of $K_n$ into three subsets $U_1, U_2, U_3$ such that $|U_1|=N/8, |U_2|=N/4, |U_3|=5N/8$. 
Edges in $U_1\cup U_2$ are labeled 1 and edges in $U_3$ are labeled 3. 
Edges between $U_i$ and $U_3$ are labeled $i$ for $i=1,2$. 

Let $\mathcal M$ be a perfect matching. 
Let $Nx_{ij}$ be the number of edges in $\mathcal M$ between $U_i$ and $U_j$ for $1\leq i<j\leq 3$. Clearly $0\leq x_{12},x_{13},x_{23}$ and $x_{12}+x_{13}\leq 1/8, x_{12}+x_{23}\leq 1/4$. 

There are $N(1/8-x_{12}-x_{13})$ points in $X_1$ covered by matching edges in $X_1$, $N(1/4-x_{12}-x_{23})$ points in $X_2$ covered by matching edges in $X_2$, and $N(5/8-x_{13}-x_{23})$ points in $X_3$ covered by matching edges in $X_3$. 
Let $Nc_{ij}$ denote the number of vertices in $U_i$ that is incident to an edge in $\mathcal M$ with label $j$. Then 
$$c_{11}=1/8, c_{12}=0, c_{13}=0;$$
$$c_{21}=1/4-x_{23}, c_{22}=x_{23}, c_{23}=0;$$
$$c_{31}=x_{13}, c_{32}=x_{23}, c_{33}=5/8-x_{13}-x_{23}.$$

For simplicity, we count the non-critical edges, and only up to an $o(N^2)$ error. 
%We omit the $n^2$ factors. 
All $\frac{25}{128}N^2$ label 3 edges are non-critical. 
All label 2 edges are between $X_2$ and $X_3$. 
Such an edge is non-critical iff both of its endpoints are covered by a label 1 or a label 2 matching edge. 
Hence, there are $(c_{21}+c_{22})(c_{31}+c_{32})N^2=\frac{x_{13}+x_{23}}{4}N^2$ such edges. % (after omitting the $n^2$ factor). 
Finally, there are several sources of non-critical label 1 edges. 
There are $c_{11}c_{31}N^2=\frac{x_{13}}{8}N^2$ between $X_1$ and $X_3$, $c_{11}c_{21}N^2=(\frac{1}{32}-\frac{x_{23}}{8})N^2$ between $X_1$ and $X_2$, $\frac{c_{11}^2}{2}N^2=\frac{1}{128}N^2$ in $X_1$, and $\frac{c_{21}^2}{2}N^2=(\frac{1}{32}-\frac{x_{23}}{4}+\frac{x_{23}^2}{2})N^2$ in $X_2$. 
Thus the number of critical edges is 
$$\left(\frac{1}{2}-\frac{25}{128}-\frac{x_{13}+x_{23}}{4}-\frac{x_{13}}{8}-\left(\frac{1}{32}-\frac{x_{23}}{8}\right)-\frac{1}{128}-\left(\frac{1}{32}-\frac{x_{23}}{4}+\frac{x_{23}^2}{2}\right)\right)N^2=$$

$$\left(\frac{15}{64}-\frac{3}{8}x_{13}+\frac{1}{8}x_{23}-\frac{1}{2}x_{23}^2\right)N^2.$$

We need to find the minimum of this function subject to the constraints $0\leq x_{12},x_{13},x_{23}$ and $x_{12}+x_{13}\leq 1/8, x_{12}+x_{23}\leq 1/4$. 
We may assume that $x_{12}=0$, as the function does not depend on $x_{12}$, and it only weakens the constraints to set $x_{12}>0$. 
Thus $0\leq x_{13},x_{23}$ and $x_{13}\leq 1/8, x_{23}\leq 1/4$. 
For a fixed $x_{23}$, it is clearly advantageous to pick the largest possible $x_{13}$, that is, $x_{13}=1/8$. 
Then the revised optimization problem is to find the minimum of $\left(\frac{3}{16}+\frac{1}{8}x_{23}-\frac{1}{2}x_{23}^2\right)N^2 = \left(\frac{3}{16}+\frac{1}{8}x_{23}(1-4x_{23})\right)N^2$ subject to the constraint $0\leq x_{23}\leq 1/4$. 
On this interval, we have $x_{23}(1-4x_{23})\leq 0$ and equality holds iff $x_{23}=0$ or $x_{23}=1/4$. 
Thus the minimum is attained at these two points, and the minimal value of $\left(\frac{3}{16}+\frac{1}{8}x_{23}(1-4x_{23})\right)N^2$ is $\frac{3}{16}N^2\sim \frac{3}{8}\binom{N}{2}$. 
It is easy to see that $x_{12}=0$ is necessary to obtain this optimum, as otherwise we lose by having to set $x_{13}<1/8$. 

Hence, the minimum ratio of critical edges is $3/8$, and it is attained by exactly two different (family of) matchings. 
We can pair up all $N/8$ vertices of $U_1$ with vertices in $U_3$, and match every other vertex within its own $U_i$. 
Alternatively, we can pair up all $N/8$ vertices of $U_1$ with vertices in $U_3$, pair up all $N/4$ vertices of $U_2$ with vertices in $U_3$, and match the remaining vertices of $U_3$ among themselves. 
\end{proof}

The second matching at the end of the proof has the exact structure that is necessary for the upper estimate to be sharp: $|X_1|=|X_3|=N/4, |X_2|=N/2$, exactly half of the edges between different $X_i$ and $X_j$ are critical, and exactly a quarter of edges in $X_2$ are critical (and no other edges). 
The quarter of the edges in $X_2$ that are critical form a clique which contain exactly one endpoint of every matching edge in $X_2$. 
It is somewhat surprising that there is also a completely different matching in the extremal structures that yield the same ratio of critical edges. 
Perhaps this is a symptom of the existence of another proof method for the upper bound, which finds a different matching that coincides with the first one in the extremal structures. 
Finding such a proof might lead to better estimates for $\gamma(\ell)$ when $\ell\geq 4$. 

Another possible way to improve the upper bound for $\ell\geq 4$ is to analyze the structure suggested by the proof. 
It seems that this structure is never optimal. 
We believe that the number $c_\ell[2]=2^{\ell-2}$ could be replaced by $c_\ell[2]=3$ when $\ell\geq 4$. 
In particular, if $\ell= 4$, this would yield the vector $c_4=(1,3,2,1)$ rather than $c_4=(1,4,2,1)$, and the upper estimate $\gamma(4)\leq 3/7\approx0.4286$ (see Lemma~\ref{lem:weights}) rather than the current one $\gamma(4)\leq 7/16=0.4375$ provided by Theorem~\ref{thm:l-labels}. 
The best lower bound we have found is %$\gamma(4)\geq \gamma'(4)\geq (9+\sqrt{3})/26 \approx 0.41277$. 
$\gamma(4)\geq \gamma'(4)\geq 5/12\approx 0.4167$. 
This is obtained by the following construction. 
Let %$a_1 = (8 - 2\sqrt{3})/52, a_2 = (9+\sqrt{3})/52, a_3 = (1+3\sqrt{3})/52, a_4 = (34 - 2\sqrt{3})/52$. 
$a_1 = 1/12, a_2 = 2/12, a_3 = 2/12, a_4 = 7/12$. 
Let $N$ be large and consider four sets $A_1, A_2, A_3, A_4$ of size $a_1N, a_2N, a_3N, a_4N$, respectively. 
(Obviously, rounding these numbers to integers does not introduce a significant error.)
Every edge incident to a vertex in $A_1$ or $A_2$ has label 1, except for edges between $A_2$ and $A_4$ which have label $2$. 
Edges in $A_3$ have label 2, those in $A_4$ have label 4, and edges between $A_3$ and $A_4$ have label 3. 
There are two optimal perfect matchings in this labeled graph yielding the above critical edge ratio $5/12$. 
This can be shown by a similar calculation as that at the end of the proof of Theorem~\ref{thm:l-labels}. 
%This can be shown for example by an elaborate Fourier-Motzkin elimination. 
%It is not unlikely that $\gamma(4)$ is strictly between $(9+\sqrt{3})/26$ and $3/7$. 

%%%%%%%%%%%%%%%%%%%%%%%%%%%%%%%%%%%%%%%%%%%%%%%%%%%%%%
\section{Cliques}\label{sec:cliques}%%%%%%%%
%%%%%%%%%%%%%%%%%%%%%%%%%%%%%%%%%%%%%%%%%%%%%%%%%%%%%%

In this section, we prove Theorem~\ref{thm:main}. 
Let $\delta \in [1,2]$ and $\ell \geq 2$ (including $\infty$). 
Assume that an $\ell$-adaptive algorithm using $n^\delta$ queries finds a clique of size $\alpha\log n$ with high probability. 

We summarize the main strategy, already introduced in \cite{feige2021tight}. 
Given a (non-negative integer) $m\leq \alpha/2$, we encode a subgraph of size at most $\alpha\log n$ by a tuple whose first $M\approx m\log n$ coordinates are in $\{1, \ldots, n^\delta\}$ and the remaining $(\alpha-2m)\log n$ coordinates are $\{1, \ldots, n\}$. 
The first $M$ entries are steps in the process, and the rest are vertices of the graph. 
Such a tuple encodes the subgraph whose vertices are the set $U$ of endpoints of the edges queried in the given $m\log n$ steps together with the remaining $(\alpha-2m)\log n$ vertices. 
In a \emph{good} tuple 
\begin{itemize}
\item the encoded graph has $N=\alpha\log n$ vertices, that is, there is no overlap in the above description of vertices, 
\item every edge was queried at some point during the process in the complete subgraph spanned by these $N$ vertices, and 
\item the $M$ edges are independent, forming a matching $\mathcal M$ in $K_N$ with a minimum number of critical edges, where the label of an edge is the round in which it was queried.
%\item the subgraph spanned by the $2m\log n$ endpoints of the $m\log n$ edges has the property that by omitting the endpoints of the last $i$ edges in the list for any $0\leq i< M$ yields a spanned subgraph with a maximum matching of size $M-i$ such that the last edge in the truncated list is contained in any maximum matching in this graph. 
\end{itemize}

%An edge $e$ in $U$ is not critical if 
%\begin{itemize}
%\item $e$ links two vertices of $U$ that belong to the same round of the adaptive process (if $\ell=\infty$, this means that it is one of the $M$ edges in the encoding), or 
%\item $e$ links endpoints of two encoded edges $e_1$ and $e_2$ that were queried in different rounds of the adaptive process, and by the time both $e_1$ and $e_2$ were queried, $e$ was already queried as well. 
%\end{itemize}

%There are at most $M^2$ non-critical edges in the fully adaptive variant $\ell=\infty$, that is, at least $\binom{\alpha\log n}{2}-m^2\log^2 n\approx (\alpha^2/2-m^2)\log^2 n$ critical ones. 
%In the $\ell$-adaptive version, we obtain the smallest number of non-critical edges if the $M$ encoded edges are equally distributed in the $r$ rounds, yielding the upper estimate $\frac{\ell-1}{\ell}M^2$ for the number of non-critical edges, and the approximate lower bound $(\alpha^2/2-\frac{\ell-1}{\ell}m^2)\log^2 n$ for the critical edges. 

\begin{proposition}\label{prop:ineqcl}
Assume that it is possible to find a clique of size $\alpha\log n$ with $\ell$ adaptive rounds and $n^\delta$ queries in $G(n, 1/2)$ for any large enough $n$ with high probability. 
Then for all $0\leq m\leq \alpha/2$ we have
$$\alpha^2/2-\alpha-2\gamma(\ell)m^2+(2-\delta)m\leq 0.$$ 
%\quad
%\begin{itemize}
%\item $\alpha^2/2-m^2-\alpha+(2-\delta)m\leq 0$ in the fully adaptive case $\ell=\infty$, and 
%\item $\alpha^2/2-\frac{\ell-1}{\ell}m^2-\alpha+(2-\delta)m\leq 0$ if we have $\ell$ adaptive rounds.
%\end{itemize}
\end{proposition}
\begin{proof}
%Let $C=C(n,\delta,\alpha,m,\ell)$ denote the number of critical edges. 
The number of possible tuples that encode a subgraph is $$(n^\delta)^{m\log n}\cdot n^{(\alpha-2m)\log n} = 2^{(\alpha+(\delta-2)m)\log^2 n}.$$ %, is greater than $2^C$. 
The probability that a given tuple encodes a clique is at most $2^{-\left(\binom{\alpha\log n}{2}-C\right)}$, where $C$ is the number of critical edges in the complete subgraph $K_N$ spanned by these $N=\alpha\log n$ vertices (with each edge labeled by the round it was queried in). 
As $C\leq \gamma(\ell)\binom{2M}{2}$, this probability is at most 
$$2^{-\left(\binom{\alpha\log n}{2}-C\right)}\leq 2^{-\left(\binom{\alpha\log n}{2}-\gamma(\ell)\binom{2m\log n}{2}\right)} \leq 2^{-(\alpha^2/2-2\gamma(\ell)m^2)\log^2n+O(\log n)}.$$ 
%Thus $C-(\alpha+(\delta-2)m)\log^2 n\leq 0$. 
%As $C\geq \gamma(\ell)\binom{M}{2}=$
%In the fully adaptive case $\ell=\infty$, we obtain $\alpha^2/2-m^2-\alpha+(2-\delta)m\leq 0$, and in the $\ell$-adaptive variant, we have $\alpha^2/2-\frac{\ell-1}{\ell}m^2-\alpha+(2-\delta)m\leq 0$.  
Using the trivial estimate that the probability of a union of events is at most the sum of the probabilities of the events yields 
$$\mathbb{P}(\text{the algorithm finds a clique of size\,\,}\alpha\log n) \leq $$
$$2^{(\alpha+(\delta-2)m)\log^2 n}\cdot 2^{-(\alpha^2/2-2\gamma(\ell)m^2)\log^2n+O(\log n)} =$$
$$2^{-(\alpha^2/2-\alpha-2\gamma(\ell)m^2+(2-\delta)m)\log^2n+O(\log n))}.$$
Hence, if $\alpha^2/2-\alpha-2\gamma(\ell)m^2+(2-\delta)m>0$ for some $0\leq m\leq \alpha/2$, then the above probability would be asymptotically 0, a contradiction. 
\end{proof}

\begin{proof}[Proof of Theorem~\ref{thm:main}]
Given $(\delta,\ell)$, we are looking for the minimum $\alpha$ such that there is a $0\leq m\leq \alpha/2$ that makes the left hand side of the inequality in Proposition~\ref{prop:ineqcl} positive. 
To this end, we first find the maximum of the expression in $m\in [0,\alpha/2]$, and then compute the minimum $\alpha$ that makes the expression positive (non-negative) for that $m$. 
In the fully adaptive case $\ell=\infty$, we have $\gamma(\ell)=1/2$, thus the derivative of the left hand side $f(m)=\alpha^2/2-\alpha-m^2+(2-\delta)m$ is $\frac{\partial}{\partial m}(\alpha^2/2-\alpha-m^2+(2-\delta)m)=-2m+(2-\delta)$, which has a unique root at $m_0=\frac{2-\delta}{2}$. 
Since $\delta\in [1,2[$, as long as $\alpha\geq 1$, the linear function $-2m+(2-\delta)$ is positive at $m=0$ and non-positive at $m=\alpha/2$. 
Hence, under the assumption $\alpha\geq 1$, the unique root  $m_0=\frac{2-\delta}{2}$ is in the interval $[0,\alpha/2]$. 

Substituting $m=m_0$ in the expression yields $f(m_0)=\alpha^2/2-\alpha+(2-\delta)^2/4$. 
The two roots of this quadratic function are $1\pm\sqrt{1-(2-\delta)^2/2}$. 
Thus the maximum $\alpha$ where the expression is non-negative is $\alpha=1+\sqrt{1-(2-\delta)^2/2}$. 
Note that it was justified to substitute $m_0$, as this $\alpha$ is indeed at least 1, therefore $m_0\in [0,\alpha/2]$. 

We argue similarly in the $\ell$-adaptive case when $\ell\in\mathbb{N}$. 
This time $f(m)=\alpha^2/2-\alpha-2\gamma(\ell)m^2+(2-\delta)m$. 
The derivative of this function is $\frac{\partial}{\partial m}(\alpha^2/2-\alpha-2\gamma(\ell)m^2+(2-\delta)m)=-4\gamma(\ell)m+(2-\delta)$. 
The unique root of this linear function is $m_0= \frac{2-\delta}{4\gamma(\ell)}$. 
Once again, $f'(0)=2-\delta>0$. 
At the other endpoint of the interval $[0,\alpha/2]$, we have $f'(\alpha/2)=-2\gamma(\ell)\alpha+(2-\delta)$. 
It is unclear whether this is necessarily non-positive at the interesting values; for example, if $\delta=1$ and $\ell=2$, making $\gamma(\ell)=1/4$, then $\alpha$ would have to be at least 2 to make this expression non-positive. 
However, as the maximum clique in $G(n,1/2)$ has size $2\log n$, this cannot provide us with a meaningful upper bound. 
So we carry on with the calculation as before, substituting $f(m_0)$ and computing the optimal $\alpha$, and then we check whether $-2\gamma(\ell)\alpha+(2-\delta)$ is non-positive in that optimum. 

Substituting $m=m_0$ in the expression yields $f(m_0)=\alpha^2/2-\alpha+\frac{(2-\delta)^2}{8\gamma(\ell)}$. 
The larger root is $\alpha=1+\sqrt{1-\frac{(2-\delta)^2}{4\gamma(\ell)}}$. 
Observe that the expression under the square root is non-negative as $\delta\in[1,2)$ and $\gamma(\ell)\geq 1/4$ for all $\ell\geq 2$. 
As noted above, this is only justified if $(2-\delta)\leq 2\gamma(\ell)\alpha$. 
If $\ell=2$ then $\gamma(\ell)=1/4$, that is, we need to check if $2-\delta\leq \frac{1}{2} + \frac{1}{2}\sqrt{1-(2-\delta)^2}$, or equivalently whether $3-2\delta\leq \sqrt{1-(2-\delta)^2}$. 
The left hand side is decreasing and the right hand side is increasing in $\delta$, and they are equal when $\delta=6/5$. 
Thus the substitution $m=m_0$ is justified for $\delta\in [6/5,2[$. 
For $\ell=2$ and $\delta\in [1,6/5]$, the best estimate we can obtain is by substituting $m=\alpha/2$ into the function $f$ and optimize for $\alpha$. 
Then $f(\alpha/2)=\alpha^2/2-\frac{\ell-1}{\ell}\alpha^2/4-\alpha+(2-\delta)\alpha/2 = \frac{3}{8}\alpha^2-\frac{\delta}{2}\alpha$ with larger root $4\delta/3$.

Now assume that $\ell\geq 3$; then $\gamma(\ell)\geq 3/8$ according to Theorem~\ref{thm:l-labels}. 
We show that in this case the above $\alpha= 1+\sqrt{1-\frac{(2-\delta)^2}{4\gamma(\ell)}}$ satisfies the inequality $2-\delta\leq 2\gamma(\ell)\alpha$ for any $\delta\in [1,2)$, thereby justifying the substitution $m=m_0$. 
Once again, the left hand side is decreasing and the right hand side is increasing in $\delta$. 
Thus it is enough to verify the inequality for $\delta=1$, that is, $1\leq 2\gamma(\ell)\left(1+\sqrt{1-\frac{1}{4\gamma(\ell)}}\right)$. 
The right hand side is increasing as a function of $\gamma(\ell)$. 
Thus it is enough to check the inequality for $\gamma(\ell)=3/8$, in which case the right hand side is approximately $1.183$. 
\end{proof}

%%%%%%%%%%%%%%%%%%%%%%%%%%%%%%%%%%%%%%%%%%%%%%%%%%%%%%
\section{Dense subgraphs}\label{sec:sense}%%%%%%%
%%%%%%%%%%%%%%%%%%%%%%%%%%%%%%%%%%%%%%%%%%%%%%%%%%%%%%

We show how the same techniques can be applied to prove estimates for the Maximum Density Subgraph Query Problem. 
\begin{proposition}\label{prop:ineqde}
Assume that it is possible to find a subgraph with edge density $\eta\in (3/4,1]$ of size $\alpha\log n$ with $\ell$ adaptive rounds and $n^\delta$ queries in $G(n, 1/2)$ for any large enough $n$ with high probability. Then for all $0\leq m\leq \alpha/2$ we have
$$(\alpha^2/2-2\gamma(\ell)m^2)(1-H(p))-\alpha+(2-\delta)m\leq 0$$ 
where $p=\frac{\eta\alpha^2/2-2\gamma(\ell)m^2}{\alpha^2/2-2\gamma(\ell)m^2}$.
%where $q=\frac{(1-\eta)\alpha^2/2}{\alpha^2/2-2\gamma(\ell)m^2}$.
\end{proposition}
%
%\begin{corollary}\label{cor:ineqds}
%For $q=\frac{(1-\eta)\alpha^2/2}{\alpha^2/2-m^2}$, we have 
%\begin{itemize}
%\item $(\alpha^2/2-m^2)(1-H(q))-\alpha+(2-\delta)m\leq 0$ in the fully adaptive case $\ell=\infty$, and 
%\item $(\alpha^2/2-\frac{\ell-1}{\ell}m^2)(1-H(q))-\alpha+(2-\delta)m\leq 0$ if we have $\ell$ adaptive rounds.
%\end{itemize}
%
%\end{corollary}
\begin{proof}
We follow the argument of the proof of Proposition~\ref{prop:ineqcl}. 
The number of possible tuples that encode a subgraph is $$(n^\delta)^{m\log n}\cdot n^{(\alpha-2m)\log n} = 2^{(\alpha+(\delta-2)m)\log^2 n}.$$ %, is greater than $2^C$. 
Let $C$ be the number of critical edges in the subgraph spanned by the $N=\alpha\log n$ vertices. 
Let $p'$ be so that $p'\big(\binom{\alpha\log n}{2}-C\big)=\eta\binom{\alpha\log n}{2}-C$. 
%That is, $p'$ is the minimum ratio of queries with positive answer among the pairs that do not constitute a critical edge so that that we need in order to find a subgraph of edge density at least $\eta$. 
That is, $p'=\frac{\eta\binom{\alpha\log n}{2}-C}{\binom{\alpha\log n}{2}-C} = 1-\frac{(1-\eta)\binom{\alpha\log n}{2}}{\binom{\alpha\log n}{2}-C}$. 
Let $p''$ be the expression we obtain by increasing $C$ in the defining formula of $p'$ to $\gamma(\ell)\binom{2M}{2}$ where $M=m\log n$. 
That is, $p''=1-\frac{(1-\eta)\binom{\alpha\log n}{2}}{\binom{\alpha\log n}{2}-\gamma(\ell)\binom{2M}{2}}$. 
Note that $p'\geq p''> 1/2$ for $n$ large enough. 
Indeed, since $C\leq \gamma(\ell)\binom{2M}{2}$, $m\leq \alpha/2$, $\eta>3/4$, and $\gamma(\ell)\leq 1/2$, we have $p'\geq p''$ and 
$$p'' = 1-\frac{(1-\eta)\binom{\alpha\log n}{2}}{\binom{\alpha\log n}{2}-\gamma(\ell)\binom{2M}{2}} = 1 - \frac{(1-\eta)(\alpha^2\log^2 n)/2 + O(\log n)}{(\alpha^2\log^2 n)/2-2\gamma(\ell)M^2 + O(\log n)} =$$
$$1 - \frac{(1-\eta)(\alpha^2\log^2 n)/2}{(\alpha^2\log^2 n)/2-2\gamma(\ell)m^2\log^2 n} + O(1/\log n)=1 - \frac{(1-\eta)\alpha^2/2}{\alpha^2/2-2\gamma(\ell)m^2} + O(1/\log n) \geq $$
$$1 - \frac{(1-\eta)\alpha^2/2}{\alpha^2/2-m^2} + O(1/\log n) \geq 1 - \frac{(1-\eta)\alpha^2/2}{\alpha^2/2-\alpha^2/4} + O(1/\log n) = 1 - \frac{(1-\eta)\alpha^2/2}{\alpha^2/4} + O(1/\log n) =$$ 
$$1-2(1-\eta)+ O(1/\log n) \rightarrow 2\eta-1 > 1/2.$$
Due to the definition of $p'$, if the encoded tuple encodes a subgraph with density at least $\eta$, then at least a proportion $p'$ of pairs not constituting a critical edge in the set of size $\alpha\log n$ turned out to be edges of $G(n,1/2)$. 
These are independent events with probability $1/2$, hence we can use the estimate that this probability is at most $2^{-\left(\binom{\alpha\log n}{2}-C\right)(1-H(p'))}$. 
If we decrease $p'$ in this estimate to $p''$, a number still at least $1/2$, then the expression increases, since the Shannon entropy function $H$ is monotone decreasing on $[1/2,1]$. 
Moreover, we can once again replace $C$ by $\gamma(\ell)\binom{2m\log n}{2}$, further increasing the upper estimate of the above probability, yielding the weaker upper bound $2^{-\left(\binom{\alpha\log n}{2}-\gamma(\ell)\binom{2m\log n}{2}\right)(1-H(p''))}$. 
Since 
$$p''=1 - \frac{(1-\eta)\alpha^2/2}{\alpha^2/2-2\gamma(\ell)m^2} + O(1/\log n) = \frac{\eta\alpha^2/2-2\gamma(\ell)m^2}{\alpha^2/2-2\gamma(\ell)m^2} + O(1/\log n) \sim p$$ 
as $n\rightarrow \infty$, and because the function $H$ is continuous, this upper bound is 
$$2^{-\left(\binom{\alpha\log n}{2}-\gamma(\ell)\binom{2m\log n}{2}\right)(1-H(p''))} = 2^{-\left((\alpha^2\log^2 n) /2-2\gamma(\ell)m^2\log^2 n\right)(1-H(q)) \cdot (1+o(1))}=$$
$$2^{-\left(\alpha^2/2-2\gamma(\ell)m^2\right)(1-H(p))\log^2 n \cdot (1+o(1))}.$$

Using the trivial estimate that the probability of a union of events is at most the sum of the probabilities of the events yields 
$$\mathbb{P}\left(\text{the algorithm finds a subgraph of size\,\,}\alpha\log n\,\,\text{with density at least}\,\, \eta\right)\leq $$
$$2^{(\alpha+(\delta-2)m)\log^2 n}\cdot 2^{-\left(\alpha^2/2-2\gamma(\ell)m^2\right)(1-H(p))\log^2 n \cdot (1+o(1))} =$$
$$2^{-((\alpha^2/2-2\gamma(\ell)m^2)(1-H(p))-\alpha+(2-\delta)m)\log^2n) \cdot (1+o(1))}.$$
Hence, if $(\alpha^2/2-2\gamma(\ell)m^2)(1-H(p))-\alpha+(2-\delta)m>0$ for some $0\leq m\leq \alpha/2$, then the above probability would be asymptotically 0, a contradiction. 
\end{proof}

We note that the assumption $\eta\in(3/4,1]$ in Proposition~\ref{prop:ineqde} could be relaxed to the condition that $p > 1/2$ for $n$ large enough. 
We have seen in the proof of Proposition~\ref{prop:ineqde} that the latter requirement is stronger than the former, as $p$ is asymptotically $p''$ and $p''>1/2$ for $n$ large enough. 
However, switching the condition $\eta\in(3/4,1]$ to ``$p> 1/2$ for $n$ large enough'' would make the already complex phrasing of Theorem~\ref{thm:dense} even more roundabout. 

\begin{proof}[Proof of Theorem~\ref{thm:dense}]
The strategy is similar to the proof of Theorem~\ref{thm:main}. 
Given $(\delta,\ell,\eta)$, we are looking for the minimum $\alpha$ such that there is a $0\leq m\leq \alpha/2$ that makes the left hand side of the inequality in Proposition~\ref{prop:ineqde} positive (non-negative). 
To this end, we first find the maximum of the expression in $m\in [0,\alpha/2]$, and then compute the minimum $\alpha$ that makes the expression non-negative for that $m$. 

Let $f(m)=(\alpha^2/2-2\gamma(\ell)m^2)(1-H(p))-\alpha+(2-\delta)m$, where $p$ is short for $\frac{\eta\alpha^2/2-2\gamma(\ell)m^2}{\alpha^2/2-2\gamma(\ell)m^2}$. 
This function is continuous and defined on the bounded, closed interval $[0,\alpha/2]$, thus it has a maximum. 
The derivative is $f'(m)=-4\gamma(\ell)m(1+\log p)+(2-\delta)$. 
As $f'(0)>0$ whenever $\delta\in [1,2)$, the maximum cannot be at the left endpoint of the domain interval. 
The maximum might be attained at the right endpoint $\alpha/2$. 
If this is not the case, then the maximum is in an inner point where the derivative is zero. 
The third derivative is $f'''(m)=\frac{32 \alpha^2 \gamma(\ell)^2 m (1 - \eta) (-16 \gamma(\ell)^2 m^4 + 3 \alpha^4 n - 4 \alpha^2 \gamma(\ell) m^2 (1 + \eta))}{(\alpha^2 - 4 \gamma(\ell) m^2)^2 (-4 \gamma(\ell) m^2 + \eta\alpha^2)^2}$. 
Every factor in this expression is clearly positive except for $-16 \gamma(\ell)^2 m^4 + 3 \alpha^4 n - 4 \alpha^2 \gamma(\ell) m^2 (1 + \eta)$. 
By using $\gamma(\ell)\leq 1/2$ and $m\leq \alpha/2$ we obtain 
$$-16 \gamma(\ell)^2 m^4 + 3 \alpha^4 n - 4 \alpha^2 \gamma(\ell) m^2 (1 + \eta) \geq (10\eta-3)\alpha^4/4 >0.$$
Hence, the derivative $f'(m)$ is convex, and in particular, it has at most two roots. 
If $f'(m)$ has only one root, then it must be the locus of the maximum of $f(m)$. 
If $f'(m)$ has two roots, then due to the convexity of $f'(m)$, the first one is a local maximum and the second one is a local minimum of $f(m)$. 
Thus only the first one can be the locus of the global maximum of $f(m)$, justifying the definition of $m_1$ in the assertion.
%The technical difficulty is that this function is more complicated than the linear function obtained in the proof of Theorem~\ref{thm:main}. 
%In particular, it is unclear whether the function $f'(m)$ has a unique root in the interval $[0,\alpha/2]$. 
%Thus if the maximum $m_0$ is an inner point of the interval $[0,\alpha/2]$, then $4\gamma(\ell)m_0\left(1+\log p\right) = 2-\delta$. 
%Hence, the other alternative is that the maximum is at $m_0=\alpha/2$, justifying the definition of $m_0$ in the assertion. 

Because $m\leq \alpha/2$, clearly we have $(\alpha^2/2-m^2)(1-H(p))-\alpha+(2-\delta)m\rightarrow \infty$ as $\alpha\rightarrow \infty$. 
Hence, the maximum $\alpha$ where the expression is non-negative for any given $m$ is the largest root. 
Given any substitution into $m$, this largest root is an upper estimate for $\alpha_{\star} \left( \delta, \ell, \eta \right)$. 
Thus the minimum of the two candidates $\min(\alpha_1, \alpha_2)$ yields the stronger upper bound. 
\end{proof}

Computer assisted numerical calculations (setting $\delta=1$) suggest that just as it is the case in the MCQP, in the MDSQP the root $m_1$ of the derivative $f'(m)$ is always in the interval $[0,\alpha/2]$ as long as $\ell\geq 3$. 
For $\ell=2$, this is not the case. For instance, if $\delta=1$, then there is a threshold $\eta_0\approx 0.936$ such that if $\eta\leq \eta_0$ then $m_1\in [0,\alpha/2]$, but if $\eta>\eta_0$ then $m_1\notin [0,\alpha/2]$. 
The calculations suggest that (given $\delta=1, \ell=2$) whenever $m_1\in [0,\alpha/2]$ (that is, if $\eta<0.936$), then $\alpha_1<\alpha_2$, thus $\alpha_1$ provides the stronger estimate.
The following graphs represent the best upper bound provided by Theorem~\ref{thm:dense}. 
See a slightly more elaborate explanation below. 

For the first diagram, we numerically approximated $\alpha_0$ for $\eta$ between $0.75$ and $0.99$, with step size $0.01$. 
For the second diagram, we numerically approximated $\alpha_0$ for $\eta$ between $0.98$ and $0.999$, with step size $0.001$. 
In both cases, we added the values for $\eta=1$ estimated by Theorem~\ref{thm:main}; that is, $1+1/\sqrt{2}\approx 1.707$ for $\ell=\infty$, $1+1/\sqrt{3}\approx 1.577$ for $\ell=3$, and $4/3$ for $\ell=2$. 
Moreover, at $\eta=1$ the trivial upper bound $2$ was added to the graph. 
The four graphs represent the trivial upper bound $2/(1-H(\eta))$, the upper bound $\alpha_1$ for $\delta=1$ and $\ell=\infty$, the upper bound $\alpha_1$ for $\delta=1$ and $\ell=3$, and the best upper bound according to the above explanation for $\delta=1$ and $\ell=2$ (that is, for $\eta\leq 0.936$ we use $\alpha_1$, and for $\eta>0.936$ we use $\alpha_2$). 
The estimates are significantly below the trivial upper bound if $\eta$ is close to 1. 
As $\eta$ approaches $0.75$, the gap slightly increases between each estimate and the trivial upper bound. 

\begin{figure}
\begin{center}
\includegraphics[scale=0.5]{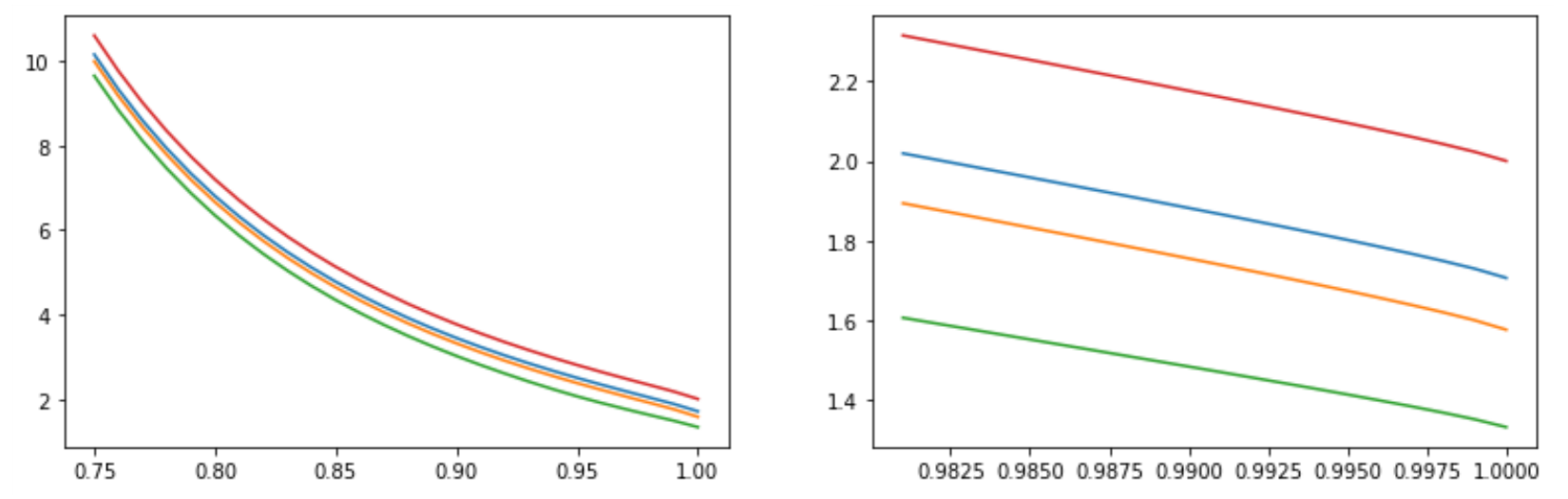}
\end{center}
\caption{Red: trivial upper bound $\frac{2}{1-H(\eta)}$. Blue: $\ell=\infty$. Orange: $\ell=3$. Green: $\ell=2$. ($\delta=1$)}
\end{figure}\label{fig:estimates}

We provide some further justification for computing $\alpha_1$ rather than $\alpha_2$ for $\ell=\infty$ and $\ell=3$. 
In the $\delta=1,\ell=\infty$ case it is easy to see why $\alpha_2$ is irrelevant. 
So assume that $m=\alpha/2$. 
As $\gamma(\infty)=1/2$, the equation in Theorem~\ref{thm:dense} simplifies to 
$(\alpha^2/4)\left(1-H\left(2\eta-1\right)\right)-\alpha/2= 0$, which has the unique solution $\alpha_2=\frac{2}{1-H\left(2\eta-1\right)}$. 
This estimate is even worse than the trivial bound $\frac{2}{1-H\left(\eta\right)}$, since $1/2\leq 2\eta-1\leq \eta$, making $H\left(\eta\right)\leq H\left(2\eta-1\right)$. 
In the $\delta=1,\ell=3$ case we only provide numerical justification. 
As $\gamma(3)=3/8$, the equation in Theorem~\ref{thm:dense} simplifies to 
$(5/16)\alpha^2\left(1-H\left((8\eta-3)/5\right)\right)-\alpha/2= 0$, which has the unique solution $\alpha_2=\frac{8}{5\left(1-H\left((8\eta-3)/5\right)\right)}$. 
This function is sketched in the following diagrams together with $\alpha_1$ (for $\delta=1,\ell=3$) in the same way as before. 

\begin{figure}
\begin{center}
\includegraphics[scale=0.5]{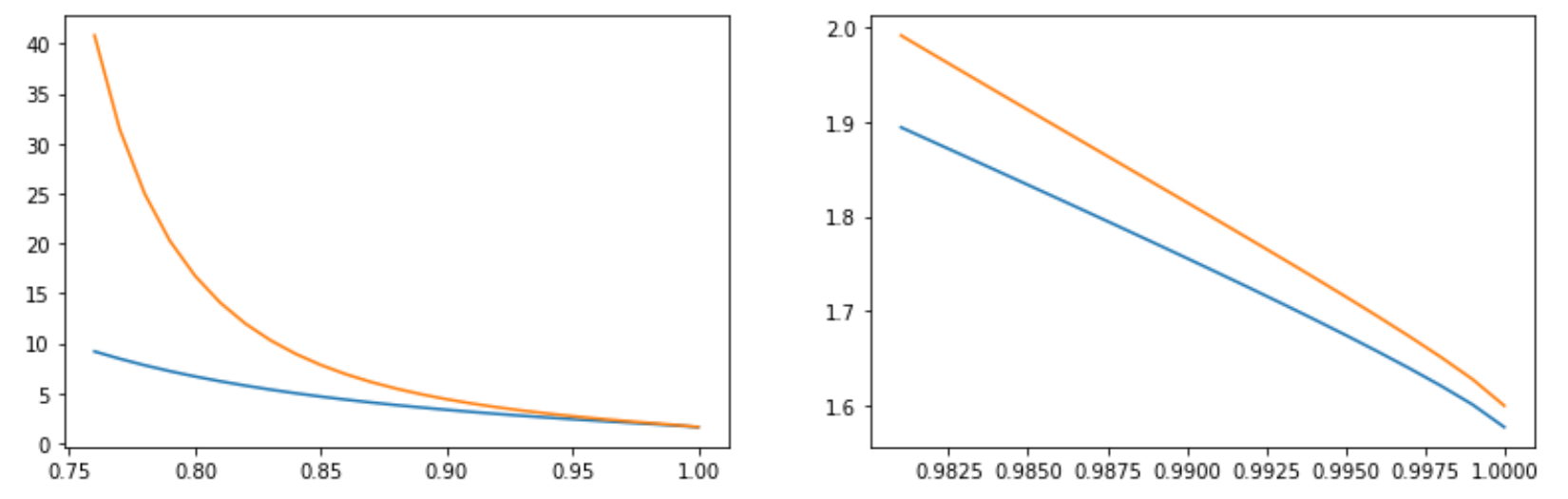}
\end{center}
\caption{Blue: $\alpha_1$ for $\delta=1, \ell=3$. Orange: $\alpha_2$ for $\delta=1, \ell=3$.}
\end{figure}\label{fig:estimatesalph}

These graphs suggest that for $\delta=1$ and $\ell=3$, the substitution $m=m_2(=\alpha/2)$ never yields a better estimate than the stationary point $m=m_1$. 

We can compare $\alpha_1$ and $\alpha_2$ for $\delta=1, \ell=2$ similarly. 
In this case, $\alpha_2=\frac{4}{3\left(1-H\left((4\eta-1)/3\right)\right)}$. 
The two diagrams are harder to distinguish than in the previous case. 
So we also provide the following numerical results (the data used to prepare the diagram on the right): 

\begin{center}
\begin{tabular}{ c c c c c c c c c }
 $\eta$ & 0.930 & 0.931 & 0.932 & 0.933 & 0.934 & 0.934 & 0.936 & 0.937 \\  
 $\alpha_1$ & 2.4116 & 2.3931 & 2.3746 & 2.3562 & 2.3380 & 2.3197 & 2.301617 & 2.28358 \\
 $\alpha_2$ & 2.4133 & 2.3943 & 2.3754 & 2.3567 & 2.3382 & 2.3198 & 2.301621 & 2.28357
\end{tabular}
\end{center}

\begin{figure}
\begin{center}
\includegraphics[scale=0.5]{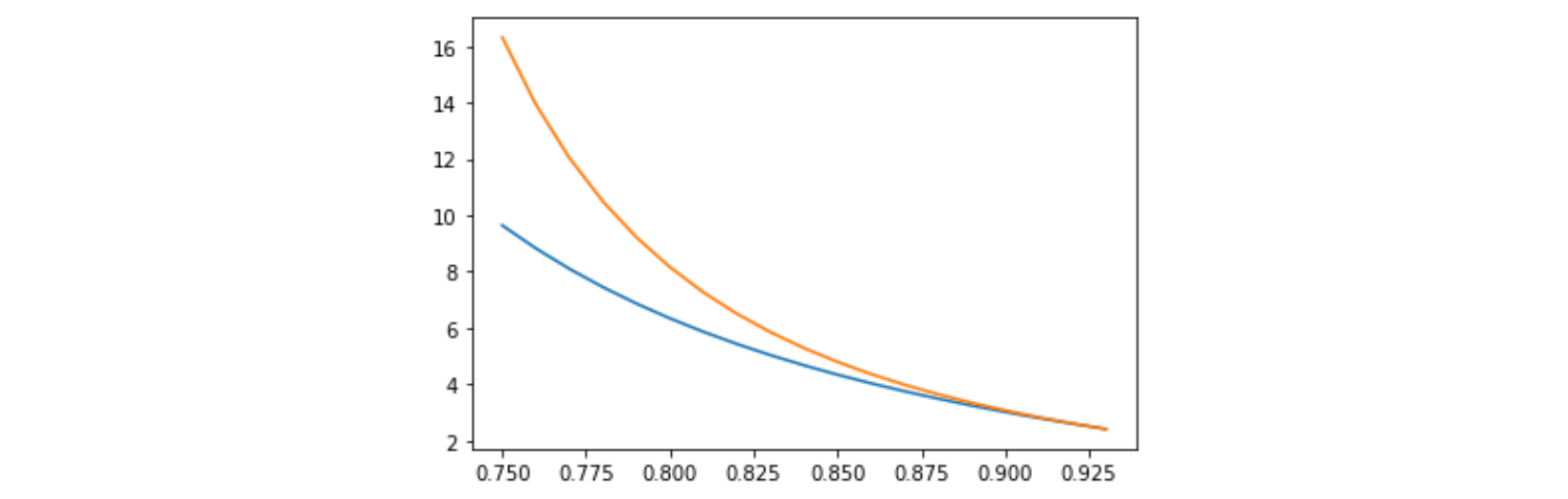}
\end{center}
\caption{Blue: $\alpha_1$ for $\delta=1, \ell=2$. Orange: $\alpha_2$ for $\delta=1, \ell=2$.}
\end{figure}\label{fig:estimatesalph212}

As a final remark, we explain why the requirement we worked with throughout the paper, namely that the algorithm should succeed with high probability, is equivalent to the seemingly weaker requirement that the probability of success should be at least $1/2$. 
%Note that in both problems, the order of magnitude of the results were $\Theta(\log n)$. 
Assume that there is an algorithm which finds a clique (or a subgraph with edge density at least $\eta$) of size $\alpha \log n$ in $G(n,1/2)$ with probability at least $1/2$. 
Then we can partition the underlying set of $G(n,1/2)$ into roughly $\log n$ subsets of size roughly $n/\log n$. 
The algorithm finds a clique (or a subgraph with edge density at least $\eta$) of size $\alpha \log (n/\log n)$ in each subset with probability at least $1/2$. 
As $\alpha \log (n/\log n) \sim \alpha \log n$, this yields an algorithm that finds a solution of size asymptotically $\alpha \log n$ with probability at least $1-1/n$. 
Indeed, the probability that the original algorithm fails in all $\log n$ subsets is at most $2^{-\log n}=1/n$. 
This argument can be generalized to some other subgraph query problems, as well. 
It does not work when we are looking for a global structure such as a Hamiltonian cycle; nevertheless, the statement itself (that the two requirements are equivalent) can be true in such a setup as well.

%%%%%%%%%%%%%%%%%%%%%%%%
%%% Acknowledgements %%%
%%%%%%%%%%%%%%%%%%%%%%%%

\section*{Acknowledgements}

This paper was initiated at the Focused Workshop on Networks and Their Limits held at the Erd\H{o}s Center (part of the Alfr\'ed R\'enyi Institute of Mathematics) in Budapest, Hungary in July 2023. 
We thank the organizers, Mikl\'os Ab\'ert, Istv\'an Kov\'acs, and Bal\'azs R\'ath, for putting together an  excellent event, and the participants of the workshop for helpful discussions. 
We are especially thankful to Mikl\'os R\'acz for posing the main problem and providing useful insights. 
The workshop was supported by the ERC Synergy grant DYNASNET 810115. 
The authors were supported by the NRDI grant KKP~138270. 

\bibliographystyle{plain}
\bibliography{refs}

%%%%%%%%%%%%%%%%
%%% Appendix %%%
%%%%%%%%%%%%%%%%

% \newpage

% \appendix

\end{document}